\documentclass{elsarticle}

\makeatletter
\def\ps@pprintTitle{%
 \let\@oddhead\@empty
 \let\@evenhead\@empty
 \def\@oddfoot{}%
 \let\@evenfoot\@oddfoot}
\makeatother
\usepackage{graphicx}
\usepackage{hyperref}
\usepackage{amsfonts}
\usepackage{epstopdf}
\usepackage{amssymb,amsmath,epsfig,multirow}
\usepackage[displaymath, mathlines,running]{lineno}
\usepackage{mathtools} 
\usepackage{color}
\usepackage{lineno,hyperref}
\usepackage{rotating}
\usepackage{makeidx}
\usepackage{lscape}
\usepackage{float}
\usepackage{epsfig}
\usepackage{amsmath,amssymb}
\usepackage{enumitem}
\usepackage{epic,eepic}
\usepackage{overpic}
\usepackage{color}
\usepackage{amsfonts}
\usepackage{graphicx}
\usepackage{enumitem}
\usepackage{pgf,tikz,pgfplots}
\pgfplotsset{compat=1.15}
\usepackage{mathrsfs}

\usepackage{algorithm}
\usepackage{algpseudocode}
\usepackage[normalem]{ulem} 

\usetikzlibrary{arrows}
\newtheorem{defn}{Definition}[section]
\newtheorem{theorem}{Theorem}[section]
\newtheorem{corollary}[theorem]{Corollary}
\newtheorem{proposition}[theorem]{Proposition}
\newtheorem{remark}{Remark}[section]
\newenvironment{proof}{\smallskip\noindent{{\it Proof.}}\hskip \labelsep}%
           {\hfill\penalty10000\raisebox{-.09em}{\large\bf\rm $\blacksquare$}\par\medskip}

\newtheorem{lemma}[theorem]{Lemma}

\DeclareMathOperator*{\argmin}{arg\,min}

\begin{document}
\begin{frontmatter}

\title{Multivariate {\color{black}Data-dependent} Partition of Unity {\color{black} based on} Moving Least Squares method}\tnotetext[label1]{This research has been supported by grant PID2023-146836NB-I00 (funded by MCIN/AEI/10.13039/501100011033) {\color{black} and CIAICO/2024/089}.}


\author[UV]{Inmaculada Garcés}
\ead{inmaculada.garces@uv.es}
\author[UPCT]{Juan Ruiz-Álvarez}
\ead{juan.ruiz@upct.es}
\author[UV]{Dionisio F. Y\'a\~nez}
\ead{dionisio.yanez@uv.es}
\date{Received: date / Accepted: date}


\address[UV]{Departamento de Matem\'aticas. Universidad de Valencia, Valencia (Spain).}
\address[UPCT]{Departamento de Matem\'atica Aplicada y Estad\'istica. Universidad  Polit\'ecnica de Cartagena, Cartagena (Spain).}

\begin{abstract}
Data approximation is essential in fields such as geometric design, numerical PDEs, and curve modeling. Moving Least Squares (MLS) is a widely used method for data fitting; however, its accuracy degrades in the presence of discontinuities, often resulting in spurious oscillations similar to those associated with the Gibbs phenomenon. This work extends the integration of MLS with the Weighted Essentially Non-Oscillatory (WENO) method and with an innovative partition of unity approach to higher dimensions. We propose a {\color{black}data-dependent} operator using the novel Non-Linear Partition of Unity \textcolor{black}{based on} Moving Least Squares method in \( \mathbb{R}^n \), which improves accuracy near discontinuities and maintains high-order accuracy in smooth regions. We demonstrate some theoretical properties of the method and perform numerical experiments to validate its effectiveness.
\end{abstract} 

\end{frontmatter}

\section{Introduction} \label{intro}



The approximation of data is a fundamental task arising in diverse areas such as computer-aided geometric design, numerical solutions of partial differential equations, and curve and surface modeling {\color{black}\cite{F, GNS, GS, H, W}}. While many approximation techniques perform well under smooth data assumptions, challenges arise when the data contain discontinuities. 

The considered problem is that of constructing an accurate approximation of an unknown function $f \colon \mathbb{R}^n \to \mathbb{R}$, based on the data values $\{f(\mathbf{x}_i)\}_{i=1}^N$ corresponding to a set of points distributed arbitrarily within an open and bounded domain $\Omega \subset \mathbb{R}^n$, where $\{\mathbf{x}_i\}_{i=1}^N \subset \Omega$. Among the established techniques, the Moving Least Squares (MLS) method is a widely adopted and robust strategy for data fitting, with applications spanning statistics and applied mathematics (see, e.g., \cite{L}). In addition, the Partition of Unity method (PUM or PU method) subdivides the domain into multiple smaller subdomains, applying the Radial Basis Function method (RBF) locally within each of them {\color{black} \cite{B,carretal,C2, cavo, CRL,CRO,DMP,W2}}. The final approximation is then obtained as a convex combination of the interpolants. When applied to data generated from a continuous function, these methods yield satisfactory results. As an illustrative case, Fig.~\ref{figintro} \textbf{(a)} displays the approximation obtained for the Franke’s function:
\begin{equation} \label{Franke}
{\color{black}f_1}(x, y) = \frac{3}{4} e^{ -\frac{(9x - 2)^2}{4} - \frac{(9y - 2)^2}{4} } + \frac{3}{4} e^{ -\frac{(9x + 1)^2}{49} - \frac{(9y + 1)}{10} } + \frac{1}{2} e^{ -\frac{(9x - 7)^2}{4} - \frac{(9y - 3)^2}{4} }  - \frac{1}{5} e^{ -(9x - 4)^2 - (9y - 7)^2 }.
\end{equation} 
Here, the MLS method is applied to a uniform grid of data points defined as $X=\{(i/n_d,j/n_d)\, \colon \, i,j=0,\dots,n_d \}$, where the discretization parameter is set to $n_d = 64$. Nonetheless, in the presence of discontinuities, this method suffers from well-known artifacts, such as the Gibbs phenomenon, that reduce its approximation quality (see \cite{ALRY}). This behavior can be clearly observed in Figs.~\ref{figintro} \textbf{(b)} and~\textbf{(c)}, where the outcome of applying the MLS method to a piecewise smooth function is illustrated\textcolor{black}{, namely}




\begin{equation}\label{frank}
{\color{black}{f_2}}(x, y) =
\begin{cases}
{\color{black} f_1}(x, y) + 1, & (x - 0.5)^2 + (y - 0.5)^2 \leq 0.25^2, \\
{\color{black} f_1}(x, y),     & \text{otherwise}.
\end{cases}
\end{equation}

As the PU method, commonly applied to RBF, is a convex linear combination of RBF interpolators, if we adjust PU to MLS to obtain PU-MLS, the expected results for data with discontinuities are similar to those shown in Figure \ref{figintro}.

\begin{figure}[H]
\centering
    \begin{tabular}{ccc} 
         \includegraphics[width=0.3\hsize]{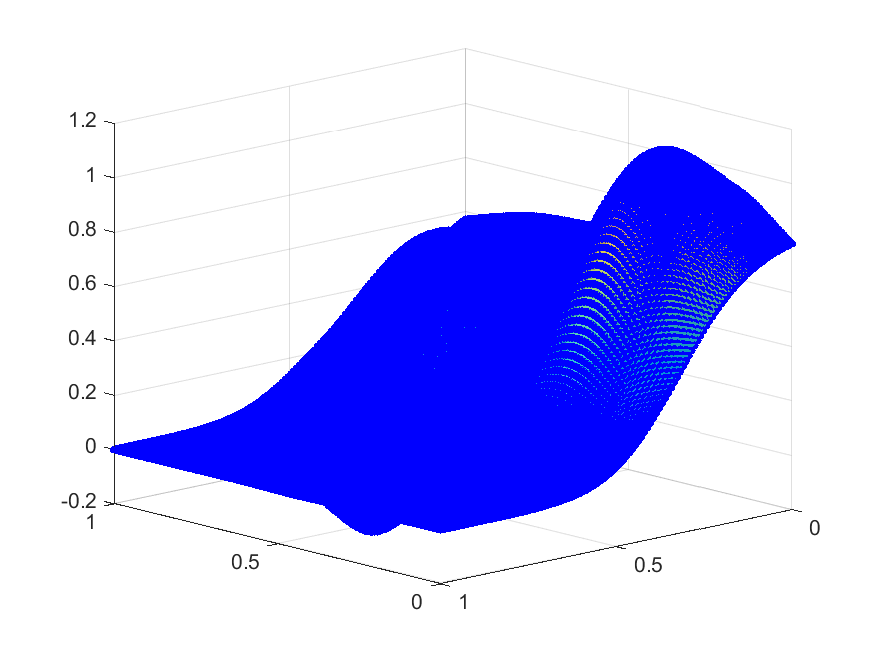} & \includegraphics[width=0.3\hsize]{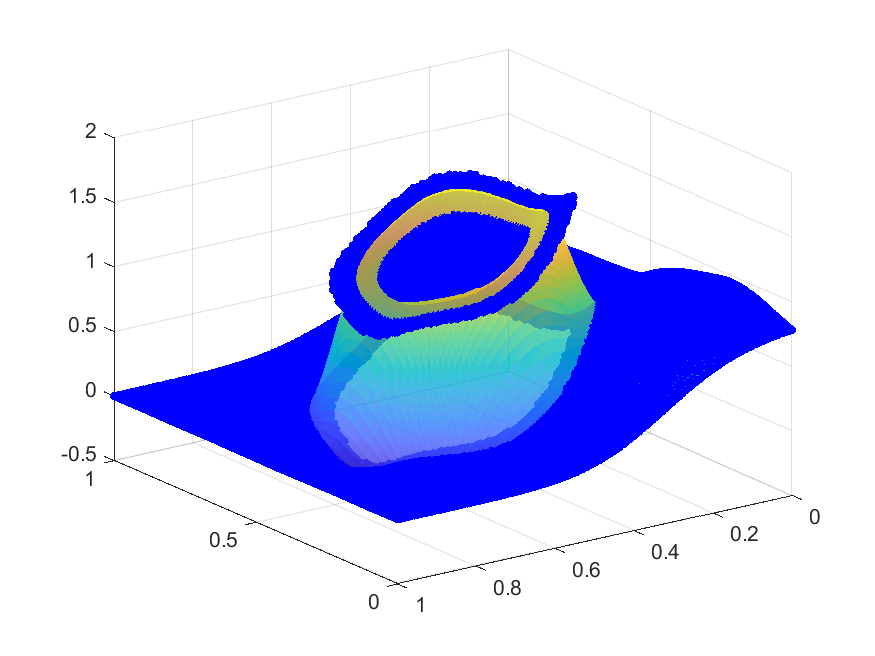} & \includegraphics[width=0.3\hsize]{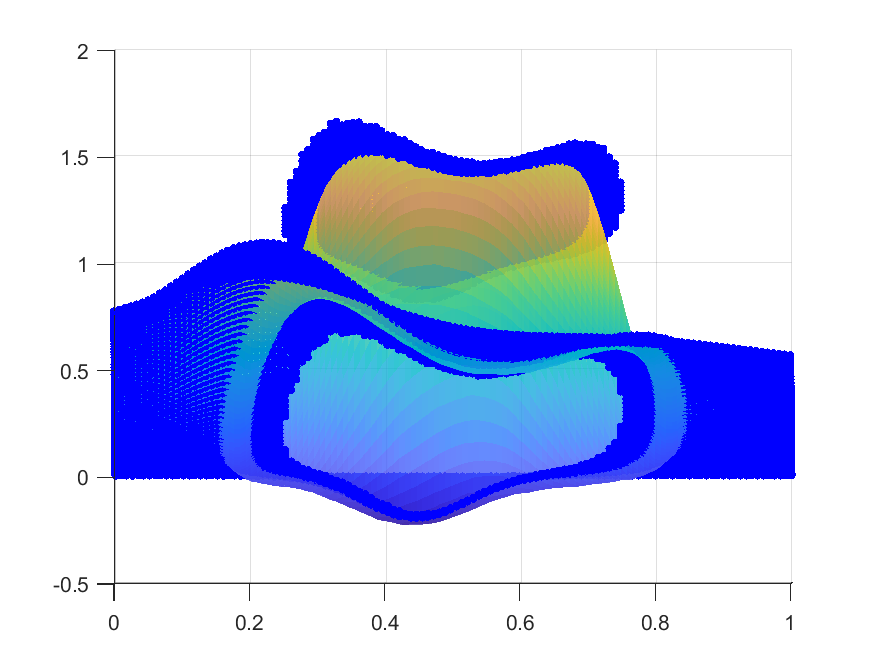}\\
         \textbf{(a)} & \textbf{(b)}  & \textbf{(c)}\\
 \end{tabular}
 \caption{\scriptsize{\textbf{(a)} MLS approximation of Franke’s function given in Eq.~\eqref{Franke}; \textbf{(b)} MLS approximation of the function ${\color{black} f_2}$ defined in Eq.~\eqref{frank}; \textbf{(c)} Rotated view of plot \textbf{(b)}.}}
 \label{figintro}
 \end{figure}

In this work, we extend previous efforts to integrate MLS with the Weighted Essentially Non-Oscillatory (WENO) method (see, e.g., \cite{BBMZ}) and a novel partition of unity framework, originally explored in one dimension in \cite{GRRY}, to higher-dimensional settings. We propose a {\color{black} data-dependent approximation operator using the Partition of Unity based on Moving Least Squares method in $\mathbb{R}^n$ (DDPU-MLS)}. This approach enhances accuracy near discontinuities while preserving high-order convergence in smooth regions. \textcolor{black}{Some new techniques have been developed in this way (see, e.g., \cite{DLNS, RRY}).} We establish theoretical properties of the method and provide numerical evidence demonstrating its performance and reliability.
This paper is divided in six sections: Firstly, we show some important properties and results of the MLS method in Section \ref{sec2}. In Section \ref{sec3}, we introduce the PUM framework and establish the theoretical results concerning its order of accuracy. We begin by recalling the fundamentals of \textcolor{black}{RBF} interpolation, and then we integrate this method with the one seen above in a linear context. The new {\color{black} DDPU}-MLS method in $\mathbb{R}^n$ will be presented in Section \ref{sec4}, together with several theoretical results that establish its main properties. Particular attention is devoted to the Gibbs phenomenon, which typically arises near discontinuities in the linear case but is mitigated in the non-linear formulation. We also analyze the approximation order in regions where the underlying data are smooth. Numerical experiments, reported in Section \ref{num}, are conducted to validate the theoretical findings. Finally, concluding remarks and perspectives are summarized in Section \ref{conc}. 



\section{\textcolor{black}{The MLS method}} \label{sec2}
 
The MLS method is a powerful algorithm designed to approximate a function $f \colon \mathbb{R}^n \to \mathbb{R}$ from values known at a scattered set of points $\{\mathbf{x}_i\}_{i=1}^N \subset \mathbb{R}^n$. This technique is also known in statistics as \textit{local polynomial regression} \cite{F}, and has applications in denoising, image processing, subdivision schemes, and the numerical solution of partial differential equations (see, e.g., \cite{carretal, davidlevin2}).


We consider the point values $ \{f(\mathbf{x}_i)\}_{i=1}^N$ as a sampling of the function $f$ at the sites $\mathbf{x}_i$, and the goal is to approximate $f(\mathbf{x})$ at a specific location $\mathbf{x} \in \mathbb{R}^n$. This value is calculated as a linear combination of the known data:
\begin{equation}\label{eq1}
    \mathcal{Q}(f)(\mathbf{x}) = \sum_{i=1}^N a_i(\mathbf{x}) f(\mathbf{x}_i),
\end{equation}
where the coefficients $\{a_i(\mathbf{x})\}_{i=1}^N$ depend on the evaluation point $\mathbf{x}$ and are computed by requiring exact reproduction of functions in a finite-dimensional polynomial space.

A weight function is now defined in terms of a function $w:[0,\infty)\to\mathbb{R}$ satisfying the following properties:
\begin{itemize}
    {\color{black}    
    \item $w$ is non-negative, compactly supported, and monotonically decreasing on $[0,c]$, with $\mathrm{supp}(w) \subseteq [0,c]$ for some constant $c>0$.
    \item $w$ attains its maximum at the origin, that is, $w(0) = \max_{r \ge 0} w(r)$, and $w$ is bounded.}
    \item There exists $p \in \mathbb{N}$ such that $w \in \mathcal{C}^p([0,\infty))${\color{black} ,
    where $w$ is understood to be extended by zero outside its support}.
\end{itemize}

Throughout this work, we assume that the set of nodes is quasi-uniformly (or uniformly) distributed. We then introduce the concept of the fill distance (see, for example, \cite{F,W}), which quantifies how well the points in $X=\{\mathbf{x}_i\}_{i=1}^N$ cover the domain $\Omega$. It is defined as

\begin{equation} \label{filldistance}
   h=h_{X, \Omega} = \sup_{\mathbf{x}\in \Omega} \min_{\mathbf{x}_j \in X} \|\mathbf{x}-\mathbf{x}_j\|_2. 
\end{equation}

We now proceed to the fundamental definition of what is referred to as a local polynomial reproduction.

\begin{defn}[\cite{W}] \label{LPR} 
Let $X = \{\mathbf{x}_i\}_{i=1}^N \subseteq \Omega$  be a set of points. We say that a family of functions $a_i : \Omega \to \mathbb{R}$, $i=1, \hdots, N$, 
achieves a local polynomial reproduction of degree $m$ on $\Omega$ if there exist constants $h_0, C_1, C_2 > 0$ such that
\begin{enumerate}
    \item[(1)] $\sum_{i=1}^N a_i(\mathbf{x}) p(\mathbf{x}_i) = p(\mathbf{x}), \, \forall p \in {\color{black} \Pi_m(\Omega)}$,
\item[(2)] $\sum_{i=1}^N |a_i(\mathbf{x})| \le C_1, \, \forall \mathbf{x} \in \Omega,$ 
\item[(3)] $a_i(\mathbf{x}) = 0 \text{ if } \|\mathbf{x} - \mathbf{x}_i\|_2 > C_2 h_{X,\Omega} \text{ and } \mathbf{x} \in \Omega,$
\end{enumerate}
for every set $X$ satisfying $h_{X,\Omega} \le h_0$.
\end{defn}

The coefficients $\{a_i\}_{i=1}^N$ can also be interpreted as generating a local polynomial reproduction; or, less rigorously, one may say that the quasi-interpolant $\mathcal{Q}(f)$ in Eq. \eqref{eq1} achieves a local polynomial reproduction.
\begin{theorem}[\cite{W}]
Suppose that $\Omega \subseteq \mathbb{R}^n$ is bounded. Let us define {\color{black} $\widetilde{\Omega}$} as the closure of $\bigcup_{x \in \Omega} B(x, C_2 h_0)$. {\color{black} Assume that $f \in C^{m+1}(\widetilde{\Omega})$}. Let us define $\mathcal{Q}(f)(\mathbf{x}) = \sum_{i=1}^N a_i f(\mathbf{x}_i)$, where $\{a_i\}_{i=1}^N$ is a local polynomial reproduction of order $m$ on $\Omega$. Then there exists a constant $C > 0$ that depends only on the constants in Definition \ref{LPR} such that
\[
|f(\mathbf{x}) - \mathcal{Q}(f)(\mathbf{x})| \le C h_{X,\Omega}^{m+1} |f|_{C^{m+1}({\color{black} \widetilde{\Omega}})},
\]
for all $X$ with $h_{X,\Omega} \le h_0$. 
The semi-norm appearing on the right-hand side is given by $|f|_{C^{m+1}({\color{black} \widetilde{\Omega}})} := \max_{|\boldsymbol{\alpha}| = m+1} \| D^{\boldsymbol{\alpha}} f \|_{\infty, {\color{black} \widetilde{\Omega}}}$, where
\[
\boldsymbol{\alpha} = (\alpha_1,\alpha_2,\dots,\alpha_n)^T \in \mathbb{N}_0^n, 
\quad |\boldsymbol{\alpha}| = \alpha_1+\alpha_2+\hdots+\alpha_n, 
\quad D^{\boldsymbol{\alpha}} f = \frac{\partial^{|\boldsymbol{\alpha}|} f}{\partial x_1^{\alpha_1}\partial x_2^{\alpha_2}{\color{black} \cdots} \, \partial x_n^{\alpha_n}}.
\] 
\end{theorem}
Therefore, the minimization problem is:

\begin{equation*} 
p_{\mathbf{x}}=\argmin_{p\in\Pi_m}\sum_{i=1}^N (f(\mathbf{x}_i)-p(\mathbf{x}_i))^2w\left(\frac{\|\mathbf{x}-\mathbf{x}_i\|_2}{h}\right).
\end{equation*}
The following definition, adapted from \cite{L}, formalizes the MLS approximation.
\begin{defn}[MLS pointwise approximation]
Let \( \{\mathbf{x}_i\}_{i=1}^N \subset \mathbb{R}^n \), and \( P = \mathrm{span}\{p_j\}_{j=1}^J \). Given a weight function \( w_i(\mathbf{x}):=w\left(\frac{\|\mathbf{x}-\mathbf{x}_i\|_2}{h}\right) \geq 0 \), the MLS approximation of \( f(\mathbf{x}) \) is defined as \( p_{\mathbf{x}}(\mathbf{x}) \), where \( p_{\mathbf{x}} \in P \) minimizes, among all $p\in P$, the weighted least-squares error functional
\begin{equation}\label{problema1}
\sum_{i=1}^N (f(\mathbf{x}_i) - p(\mathbf{x}_i))^2 w_i(\mathbf{x}).
\end{equation}
\end{defn}
That is, \( f(\mathbf{x}) \) is approximated by evaluating at \( \mathbf{x} \) the polynomial \( p_{\mathbf{x}} \) that best fits the data \( \{f(\mathbf{x}_i)\}_{i=1}^N \) around \( \mathbf{x} \) in a weighted least-squares sense. To solve \eqref{problema1}, we use Theorem 4.3 in \cite{W}.

\begin{theorem}[\cite{W}] \label{coeff}
Suppose that for each $\mathbf{x}\in \mathbb{R}^n$ the set $\{\mathbf{x}_i\}_{i=1}^N \subset \mathbb{R}^n$ is $\Pi_m(\mathbb{R}^n)$-unisolvent (the zero polynomial is the only polynomial of $\Pi_m(\mathbb{R}^n)$ that vanishes in all of them). In this situation, problem \eqref{problema1} has a unique solution. The solution $\mathcal{Q}(f)(\mathbf{x}) = p_{\mathbf{x}}(\mathbf{x})$ can be expressed in the form $\mathcal{Q}(f)(\mathbf{x}) = \sum_{i=1}^N a^{*}_i(\mathbf{x})f(\mathbf{x}_i),$ with coefficients $a^{*}_i(\mathbf{x})$ defined as the solution to the minimization of the quadratic functional $\sum_{i=1}^N a_i(\mathbf{x})^2\frac{1}{w_i(\mathbf{x})}$ subject to the interpolation conditions $\sum_{i=1}^N a_i(\mathbf{x}) p(\mathbf{x}_i) = p(\mathbf{x}),\, \forall p \in \Pi_m(\mathbb{R}^n)$. 
\end{theorem}
We can express the operator \(\mathcal{Q}(f)\) in matrix form as
\begin{equation}\label{solucionMLS_pointwise}
\mathcal{Q}(f)(\mathbf{x}) = (f(\mathbf{x}_1), \ldots, f(\mathbf{x}_N)) D E (E^T D E)^{-1} (p_1(\mathbf{x}), \ldots, p_J(\mathbf{x}))^T,
\end{equation}
where \(E \in \mathbb{R}^{N \times J}\) and \(D \in \mathbb{R}^{N \times N}\) are defined as
\begin{equation}\label{construccionE_pointwise}
E_{i,j} = p_j(\mathbf{x}_i), \quad 1 \leq i \leq N, \; 1 \leq j \leq J, \qquad 
D = \mathrm{diag}(w_1(\mathbf{x}), \ldots, w_N(\mathbf{x})).
\end{equation}
{\color{black} We note that although some weights $w_i(\mathbf{x})$ may vanish, the matrix $E^T D E$ remains invertible. Indeed, only the nodes $\mathbf{x}_i$ such that $w_i(\mathbf{x})>0$ contribute to the matrix, and by the local polynomial reproduction assumption (Definition~\ref{LPR}), these nodes form a unisolvent set for the polynomial space of degree at most $m$. As a consequence, $E^T D E$ is symmetric positive definite and therefore invertible.}

Assuming that \(\mathrm{rank}(E) = J\), the expression above provides the unique solution for the local approximation at the evaluation point \(\mathbf{x}\). This approach further assumes that the data are sufficiently smooth to be well approximated by a polynomial within a neighborhood of each point \(\mathbf{x}\). We refer to this setting as the \emph{linear} MLS approach. Based on this formulation, the separation distance and quasi-uniformity can then be defined using Eq.~\eqref{filldistance}.

\begin{defn}[\cite{W}]
 Consider a finite set of points $X = \{\mathbf{x}_1, \dots, \mathbf{x}_N\} \subset \Omega$. The separation distance of $X$ is given by
 $$q_X:=\frac{1}{2} \min_{i\neq j} \|\mathbf{x}_i-\mathbf{x}_j\|_2.$$
 Moreover, the set $X$ is considered quasi-uniform, characterized by a constant {\color{black} $c_{\text{qu}} \ge 1$}, whenever
 \begin{equation}\label{acot}
    q_X \leq h_{X, \Omega} \leq c_{\text{qu}} q_X. 
 \end{equation}
\end{defn}

The separation distance provides the largest possible radius for two balls centered at different data locations to be essentially disjunct. The notion of quasi-uniformity should be understood in relation to multiple collections of data sites. Specifically, we consider a sequence of sets whose elements \textcolor{black}{gradually cover} the domain $\Omega$ more densely. In this setting, it is crucial that condition \eqref{acot} holds for every set in the sequence, using the same constant $c_{\text{qu}}$.



We assume that the domain $\Omega \subset \mathbb{R}^n$ fulfills the interior cone condition, following the formulation in \cite{F, RRY, W}. This means that there exist an angle $\theta \in (0, \pi/2)$ and a radius $r > 0$ for which the following holds: for every point $\mathbf{x} \in \Omega$, there is a unit direction vector $\boldsymbol{\xi}(\mathbf{x})$ such that the cone
$$C=\{\mathbf{x}+\lambda \mathbf{y}\, \colon \, \mathbf{y}\in \mathbb{R}^n, \, \|\mathbf{y}\|_2=1, \, \mathbf{y}^T\boldsymbol{\xi}(\mathbf{x})\geq \cos\theta, \, \lambda \in [0,r]\}$$
is contained in $\Omega$ (see Figure \ref{cone}).
\begin{figure}[H]
\centering 
\includegraphics[width=0.35\hsize]{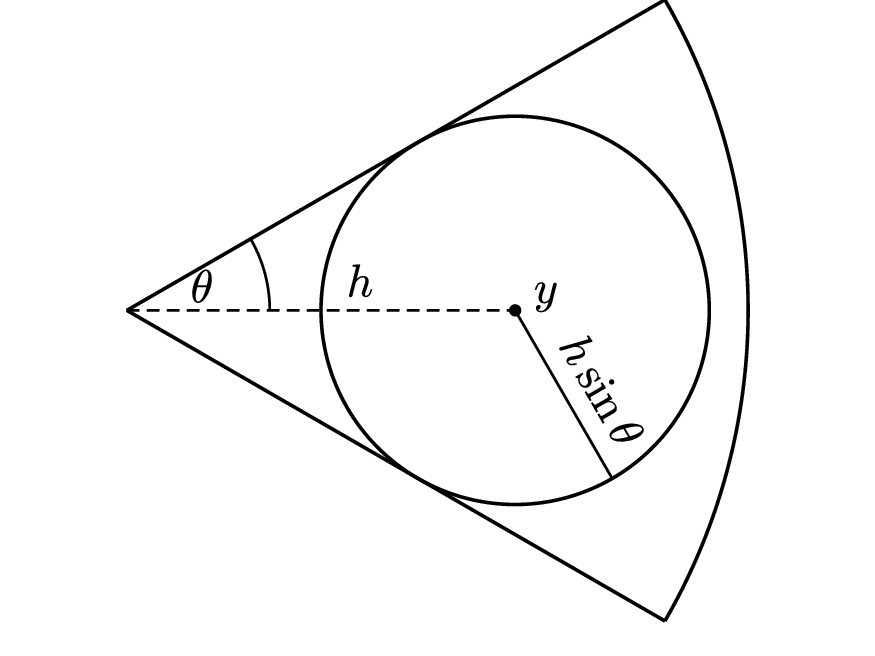}
\caption{Ball in a cone.}
\label{cone}
\end{figure}

The lemma established by Wendland in \cite{W} shows that any ball fulfills the interior cone condition. 
This property becomes relevant in our numerical experiments, particularly when constructing a specific version of the partition of unity method.
\begin{lemma}[\cite{RRY, W}]
Any ball with radius $\delta > 0$ fulfills the interior cone condition, defined by the same radius $\delta$ and an angle of $\theta = \pi/3$.
\end{lemma}


Having established these preliminary results, we can now determine the corresponding order of accuracy, following the statements of Theorem~4.7 and Corollary~4.8 in \cite{W}.

\begin{theorem}[\cite{W}] \label{orderL}
Assume that $\Omega \subset \mathbb{R}^n$ is a compact set fulfilling the interior cone condition for some angle $\theta \in (0, \pi/2)$ and radius $r > 0$. Fix an integer $m \in \mathbb{N}$, and let $h_0, C_1,$ and $C_2$ be positive constants depending only on $m$, $\theta$, and $r$. Suppose that $X=\{\mathbf{x}_1, \dots, \mathbf{x}_N\} \subseteq \Omega$ satisfies \eqref{acot} and $h_{X,\Omega} \leq h_0$. Denote by ${\color{black} \widetilde{\Omega}}$ the closure of $\bigcup_{\mathbf{x} \in \Omega} B(\mathbf{x}, 2 C_2 h_0)$. Assume that the local basis functions $a_i^*(\mathbf{x})$ are constructed as in Th. \ref{coeff} and have supports contained in $B(\mathbf{x}, \delta)$, i.e.\ they depend on $h_{X,\Omega}$ through $\delta:= 2C_2 h_{X,\Omega}$. Then, for all $\mathbf{x} \in \Omega$, the following properties hold:
\begin{enumerate}
    \item[(1)] $\sum_{i=1}^N a_i^*(\mathbf{x}) \, p(\mathbf{x}_i) = p(\mathbf{x}),\, \forall p \in P$,
    \item[(2)] $\sum_{i=1}^N |a_i^*(\mathbf{x})| \le \widetilde{C}_1$, 
    \item[(3)] $a_i^*(\mathbf{x}) = 0 \, \, \text{if } \, \, \|\mathbf{x} - \mathbf{x}_i\|_2 > \widetilde{C}_2 h_{X,\Omega}$,
\end{enumerate}
for certain constants $\widetilde{C}_1, \widetilde{C}_2 > 0$. 
\end{theorem}

\begin{corollary}[\cite{W}]
Under the assumptions of Theorem~\ref{orderL}, one can identify an explicitly computable constant $c > 0$ for which, for every function $f \in \mathcal{C}^{m+1}({\color{black} \widetilde{\Omega}})$, the approximation error satisfies the following bound:
\begin{equation*}
    \|f-\mathcal{Q}(f)\|_{\infty, \Omega} \leq ch^{m+1}_{X, \Omega}|f|_{\mathcal{C}^{m+1}({\color{black} \widetilde{\Omega}})}.
\end{equation*}
\end{corollary}

However, when the data exhibits discontinuities or strong gradients, linear MLS may introduce artifacts such as Gibbs oscillations. To address this, the WENO (Weighted Essentially Non-Oscillatory) method has been developed \cite{S}. WENO builds non-linear combinations of local polynomial approximations and includes smoothness indicators to detect non-smooth behavior in the data. We will look at this in more detail in Section \ref{sec4}.



\section{\textcolor{black}{The PU method}}\label{sec3}

The \textcolor{black}{PUM}, often referred to simply as the PU method, was introduced to facilitate efficient numerical computations, particularly in the context of meshfree approximation techniques. The main idea involves breaking down a complex problem into several smaller, manageable parts while maintaining the desired approximation accuracy. Let $\Omega \subset \mathbb{R}^n$ be an open and bounded domain, which is covered by a collection of $M$ overlapping subdomains $\Omega_j$ such that \cite{C2, RRY, W2}:

\begin{itemize}
    \item The collection of subdomains fully covers the domain $\Omega$, that is,
    $$\Omega \subseteq \bigcup_{j=1}^M \Omega_j.$$
    
    \item There exists a uniform constant $K$ ensuring that each point $\mathbf{x} \in \Omega$ belongs to no more than $K$ subdomains $\Omega_j$.
    
    \item There exist a positive constant $C_r$ and an angle $\theta \in (0, \pi/2)$ such that each intersection $\Omega_j \cap \Omega$ fulfills an interior cone condition characterized by the angle $\theta$ and the radius $r = C_r h_{X,\Omega}$.
    
    \item The local fill distance $h_{X_j,\Omega_j}$, where $X_j = \Omega_j \cup X$, is uniformly controlled by the global fill distance $h_{X,\Omega}$.
\end{itemize}


Such a configuration is referred to as a regular covering of $(\Omega, X)$, see \cite{W}. Based on this covering, we construct a set of continuous, non-negative, compactly supported functions $\{w_j\}_{j=1}^M,$ with $w_j \, \colon \, \Omega_j \to \mathbb{R}$, satisfying
\[
\sum_{j=1}^M w_j(\mathbf{x}) = 1 ,\, \forall \mathbf{x}\in \Omega \quad \text{ and } \quad \text{supp}(w_j) \subseteq \Omega_j.
\]
This family of functions is called a partition of unity. The global approximation is then defined as the weighted combination 
\[
\mathcal{Q}_{\text{PU}}(f)(\mathbf{x}) = \sum_{j=1}^M w_j (\mathbf{x}) \mathcal{Q}_j(f)(\mathbf{x}).
\]
Consider a bounded domain $\Omega \subseteq \mathbb{R}^n$, and let $\{\Omega_j\}_{j=1}^M$ denote a collection of open and bounded subsets satisfying $\Omega \subseteq \bigcup_{j=1}^M \Omega_j$.

\begin{defn}[\cite{W}]
Consider a family of nonnegative functions $\{w_j\}_{j=1}^M$, where each $w_j \in \mathcal{C}^k(\mathbb{R}^n)$. We say that this family forms a $k$-stable partition of unity associated with the covering $\{\Omega_j\}_{j=1}^M$ whenever the following properties are satisfied for every index $j = 1, \dots, M$:
\begin{enumerate}
    \item The support of each function is contained in its corresponding set: \quad $\operatorname{supp}(w_j) \subseteq \Omega_j.$ 
    \item The functions form a partition of unity on $\Omega$:\quad $\sum_{j=1}^M w_j \equiv 1 \quad \text{on } \Omega.$ 
    \item For each multi-index $\alpha \in \mathbb{N}_0^n$ satisfying $|\alpha| \leq k$, there exists a constant $C_\alpha > 0$ such that
    \[
    \|D^\alpha w_j\|_{L^\infty(\Omega_j)} \leq \frac{C_\alpha}{\delta_j^{|\alpha|}}, \quad \text{for all } 1 \leq j \leq M \quad \text{and} \quad \delta_j = \mathrm{diam}(\Omega_j) = \sup_{x,y \in \Omega_j} \|x - y\|_2.
    \]
\end{enumerate}
\end{defn}


The PU method divides the global domain into a set of smaller overlapping subdomains, on each of which the RBF approach is locally applied. Next, we are going to integrate this method into the MLS method, seen in Section \ref{sec2}.

\subsection{\textcolor{black}{The RBF method}}
The RBF interpolation method has been the subject of an extensive study in the literature (see, for example, \textcolor{black}{\cite{B, C2, F, W}}). In the present work, we shall primarily adopt the notation established in \cite{cavo}. The RBF method constructs an approximation of the target function by expressing it as a linear combination of radial basis functions centered at a given set of data points, denoted by $X=\{\mathbf{x}_i\}_{i=1}^N \subset \Omega$. Each point $\mathbf{x}_i$ is associated with a function value $f_i = f(\mathbf{x}_i)$, where $f \, \colon \, \Omega \to \mathbb{R}$ represents an unknown function sampled at the nodes. The goal of the scattered data interpolation problem is to determine an operator 
$\mathcal{Q}(f)\, \colon \, \Omega \to \mathbb{R}$ that exactly reproduces the known data at the corresponding locations, that is, 

$$\mathcal{Q}(f)(\mathbf{x}_i)=f(\mathbf{x}_i), \quad i=1, \dots, N.$$
Our objective is to minimize the distance between $\mathcal{Q}(f)$ and $f$ with respect to a chosen norm. To begin, we introduce a radial basis function $\Phi \colon \Omega \to \mathbb{R}$, defined through a univariate function $\phi \colon [0, \infty[ \to \mathbb{R}$ such that
$$\Phi(\mathbf{x}) = \phi(\|\mathbf{x}\|_2), \quad \forall \mathbf{x} \in \Omega {\color{black} ,}$$
where $\| \cdot \|_2$ denotes the Euclidean norm in $\mathbb{R}^n$. We define $\Phi_i(\mathbf{x})=\phi(\|\mathbf{x}-\mathbf{x}_i\|_2), \, i=1, \dots, N$ and employ a radial basis function expansion to express the interpolation problem:
\begin{equation}\label{interpolant}
    \mathcal{Q}(f)(\mathbf{x})=\sum_{i=1}^N c_i \Phi_i(\mathbf{x}) = \sum_{i=1}^N c_i \phi(\|\mathbf{x}-\mathbf{x}_i\|_2), \quad \mathbf{x}\in \Omega,
\end{equation}
where $c_i$ are the coefficients determined by the interpolation system $A\mathbf{c}=\mathbf{f}$, with
\begin{equation*}
    A=\begin{bmatrix}
        \phi(\|\mathbf{x}_1-\mathbf{x}_1\|_2) &  \phi(\|\mathbf{x}_1-\mathbf{x}_2\|_2) & \dots &  \phi(\|\mathbf{x}_1-\mathbf{x}_N\|_2)\\
        \phi(\|\mathbf{x}_2-\mathbf{x}_1\|_2) &  \phi(\|\mathbf{x}_2-\mathbf{x}_2\|_2) & \dots &  \phi(\|\mathbf{x}_2-\mathbf{x}_N\|_2)\\ 
        \vdots & \vdots & \ddots & \vdots\\
         \phi(\|\mathbf{x}_N-\mathbf{x}_1\|_2) &  \phi(\|\mathbf{x}_N-\mathbf{x}_2\|_2) & \dots &  \phi(\|\mathbf{x}_N-\mathbf{x}_N\|_2)\\
    \end{bmatrix}, \, \mathbf{c}=\begin{bmatrix}
        c_1\\
        c_2\\
        \vdots\\
        c_N
    \end{bmatrix}, \, \mathbf{f}= \begin{bmatrix}
    f_1\\
    f_2\\
    \vdots\\
    f_N
    \end{bmatrix}.
\end{equation*}

The matrix $A$ is evidently symmetric. {\color{black} In order to guarantee the uniqueness of the solution of the associated linear system, it is necessary to ensure that $A$ is positive definite. This property depends on the choice of the radial basis function $\phi$, which is required to satisfy suitable positive definiteness conditions. To this end, we recall the following definition (see \cite{F}).}

\begin{defn}[\cite{F}] 
 A real-valued and continuous function $\Phi : \mathbb{R}^n \to \mathbb{R}$ is defined to be positive definite over $\mathbb{R}^n$ if, for every collection of $N$ distinct points $\{\mathbf{x}_i\}_{i=1}^N \subset \mathbb{R}^n$ and for any coefficient vector $c = [c_1, \ldots, c_N] \in \mathbb{R}^N$, the inequality below is satisfied:
 \begin{equation}\label{posdef}
     \sum_{j=1}^N \sum_{i=1}^N c_j c_i \Phi_i(\mathbf{x}_j) \geq 0.
 \end{equation}
 Moreover, $\Phi$ is said to be strictly positive definite on $\mathbb{R}^n$ if the expression in \eqref{posdef} equals zero only when $\mathbf{c=0}$.
\end{defn}

Several choices of the function $\phi$ are known in the literature to yield strictly positive definite kernels. A selection of commonly used examples is reported in Table~\ref{tabla1nucleos} (see, e.g., \cite{F}). In this table, the cutoff function $(\cdot)_+ : \mathbb{R} \to \mathbb{R}$ is defined as
 \begin{equation*}
 (x)_+=\begin{cases}
 x, & x\geq 0,\\
 0, & x<0.
 \end{cases}
 \end{equation*}

\begin{table}[H]
\centering
\begin{tabular}{lll}
\hline
$\phi(  \,r)$ & RBF & Denoted as:                                                         \\ \hline
$e^{-  r^2}$ & Gaussian $\mathcal{C}^\infty$ & G                                          \\
$\left(1 + r^2 \right)^{-1/2}$ & Inverse MultiQuadratic $C^\infty$ & IMQ  \\
$e^{-  r}$ & Mat\'ern $\mathcal{C}^0$ & M0 \\
$e^{-  r} \left( 1 +   r \right)$ & Mat\'ern $\mathcal{C}^2$ & M2 \\
$e^{-  r} \left( 3 + 3  r +   r^2 \right)$ & Mat\'ern $\mathcal{C}^4$ & M4 \\
$(1 -   r)^2_+$ & Wendland $\mathcal{C}^0$ & W0 \\
$(1 -   r)^4_+ \left( 4  r + 1 \right)$ & Wendland $\mathcal{C}^2$ & W2 \\
$(1 -   r)^6_+ \left( 35  r^2 + 18  r + 3 \right)$ & Wendland $\mathcal{C}^4$ & W4 \\ \hline
\end{tabular}
\caption{Selected radial basis functions (RBFs).}\label{tabla1nucleos}
\end{table}

A remarkable feature of the interpolant introduced in Eq.~(9) lies in its smoothness, as it can be represented through a linear mixture of smooth basis functions $\Phi_i$. Consequently, the continuity property of the interpolation operator is inherited from that of the radial function $\phi(\|\cdot\|_2)$.



\subsection{Integrating MLS with the Partition of Unity method in $\mathbb{R}^n$}

We suppose an open domain $\Omega \subseteq \mathbb{R}^n$, some data points $\{\mathbf{x}_i\}_{i=1}^N$ with $\mathbf{x}_i=(x^i_1, \dots, x^i_n)$ their components and $h$ the fill distance, and some associated set $\{f_i = f(\mathbf{x}_i)\}_{i=1}^N$ with $f : \Omega \to \mathbb{R}$ an unknown function. Let $\mathbf{x} \in \mathbb{R}^n$ be a point of the domain, and we suppose that the compact support of the weight function $w$ is of size a constant $c$. We consider that $\{\tilde{\mathbf{x}}_k\}_{k=1}^M \subset \mathbb{R}^n$ are some center points with {\color{black} $\tilde{\mathbf{x}}_k = (\tilde{x}_1^k, \dots, \tilde{x}_n^k)$} their components and call 
\[
\Omega_k = {\color{black} B}(\tilde{\mathbf{x}}_k, {\color{black}\varepsilon(h)})=\{ \mathbf{x}\in \mathbb{R}^n : \|\mathbf{x}-\tilde{\mathbf{x}}_k\|_2<{\color{black}\varepsilon(h)} \}, \quad k = 1, \dots, M,
\] 
where $M$ denotes the number of subdomains employed in the partition of unity {\color{black}(in general, $M$ is not required to
coincide with $N$)}, and \( {\color{black}\varepsilon(h)} > 0 \) determines the radius of the balls.

The MLS problem for each \( k \) in $\mathbb{R}^n$ is:
\[
p_k = \underset{p \in \Pi_d}{\arg\min} \sum_{\mathbf{x}_i \in \Omega_k} \left( f(\mathbf{x}_i) - p(\mathbf{x}_i) \right)^2 w\left({\color{black}\mu_k} \frac{\|\mathbf{x}-\mathbf{x}_i\|_2}{h}\right),
\]
{\color{black}where $\mu_k>0$ is a constant (similar to the shape parameter, see \cite{W})}. 

The set of centers $\{\tilde{\mathbf{x}}_k\}_{k=1}^M$ must be chosen so as to satisfy the following conditions:
\begin{enumerate}
    \item The union of all subdomains should cover the entire domain $\Omega$, that is,
$$\Omega\subseteq \bigcup_{k=1}^M\Omega_k.$$
\item The subdomains $\Omega_k$ are required to overlap slightly.
\item Each subdomain must contain more data points than the number of polynomial terms, i.e.,
$$|\Omega_k\cap\{\mathbf{x}_i\}_{i=1}^N|> \binom{d+n}{n},\quad \forall k=1,\hdots,M.$$
\item For each point $\mathbf{x} \in \mathbb{R}^n$, the count of subdomains that include $\mathbf{x}$ is limited by a global constant $K_1$, that is,
$$\sum_{k=1}^M \chi_{\Omega_k}(\mathbf{x}) \leq K_1, \quad \text{for all}\,\,\mathbf{x}\in \mathbb{R}^n.$$
\end{enumerate}
Thus, a partition of unity method can be designed 
\begin{equation}\label{PUM}
\mathcal{Q}_\text{PU}(f)(\mathbf{x})=\sum_{k=1}^M \theta_k (\mathbf{x})p_k(\mathbf{x}),
\end{equation}
being
$$\theta_k(\mathbf{x})=\frac{\delta_k(\mathbf{x})}{\sum_{j=1}^M\delta_j(\mathbf{x})},\quad \text{with} \quad \delta_k(\mathbf{x})=w\left({\color{black}\mu_k}\frac{\|\mathbf{x}-\tilde{\mathbf{x}}_k\|_2}{h}\right).$$



\section{\textcolor{black}{The DDPU-MLS method}} \label{sec4}


In order to introduce non-linearity in our method, we are going to use a well-known tool, WENO. The WENO scheme has become widely recognized as an effective approach for data reconstruction, particularly in problems involving discontinuities (see, e.g.,~\cite{S}). It builds several polynomial approximations on different stencils and combines them using non-linear weights. \textcolor{black}{For the sake of clarity and simplicity of exposition, we briefly recall the
main ideas of the WENO methodology in the one-dimensional setting.} Consider a function $f(x)$ whose value we wish to approximate at {\color{black} the mid-point of the interval $(x_{i-1}, x_i)$, denoted by $x_{i-\frac{1}{2}}$}. The WENO reconstruction of $f$ at {\color{black}$x_{i-\frac{1}{2}}$} can be expressed as
$$\hat{f}({\color{black}x_{i-\frac{1}{2}}})=\sum_{k=0}^{r-1} \beta_k p_k({\color{black}x_{i-\frac{1}{2}}}),$$
where $r$ denotes the number of stencils, $p_k$ are the local polynomial approximations associated with each stencil, and $\beta_k$ are the corresponding non-linear weights. These weights are defined as
$$\beta_k=\frac{\alpha_k}{\sum_{l=0}^{r-1} \alpha_l}, \quad \text{with } \alpha_k=\frac{C_k}{(\epsilon+I_k)^t},$$
where $C_k$ denotes the linear weights, $\epsilon$ is a small positive parameter introduced to prevent division by zero, and $t$ is generally assigned the value $2$ to improve precision within smooth areas. The smoothness indicators $I_k$ are defined by:

\begin{equation*}
I_k = \sum_{l=1}^{r-1} \int_{x_{i-1/2}}^{x_{i+1/2}} \Delta x^{2l-1} \left( \frac{d^l}{dx^l} {\color{black} p}_k(x) \right)^2 dx,
\end{equation*}
where $\Delta x$ represents the grid spacing. These values mark whether one node is close to the discontinuity or not.

{\color{black} In this section, we briefly review the classical WENO methodology developed in \cite{AMR}, which introduces a genuinely multivariate, non-separable two-dimensional WENO interpolation as an extension of the one-dimensional scheme, while preserving the same order of accuracy in smooth regions without relying on tensor-product constructions. This approach was later improved in \cite{MRSY} by constructing a general multivariate WENO interpolant with progressive order of accuracy. 

We consider the mid-point of the hypercube where we want to approximate the function
$\mathbf{x}^{*}=(x^*_1,\hdots,x^*_n)\in \prod_{j=1}^n [x^{(i_j-1)}_{j},x^{(i_j)}_{j}],$ with $\mathbf{x}^{*}=\mathbf{x}^{*}_{\frac{1}{2}}$. The construction follows closely the one-dimensional case. We design a convex combination of lower-degree interpolants of degree $r$. Let $\mathbf{k}=(k_1,\hdots, k_n)$ and define
$$p_0^{2r-1}(\mathbf{x}^{*}_{\frac{1}{2}})=\sum_{\mathbf{k}\in \{0,1,\hdots,r-1\}^n} C^r_{ \mathbf{k}}p_{\mathbf{k}}^r(\mathbf{x}^{*}_{\frac{1}{2}}),$$
where $C^r_{\mathbf{k}}=\prod_{j=1}^n C^r_{k_j}$ are the optimal weights, , and $C^r_{k_j}$, for $k_j = 0,\hdots,r-1$ and $j = 1,\hdots,n$, denote the corresponding one-dimensional optimal weights, whose explicit formula is given in \cite{ABM}: 
\begin{equation}\label{opt_w}
{C}_{k}^r=\frac{1}{2^{2r-1}}
\binom{2r}{2k+1}
,  \quad k=0,\hdots, r-1.
\end{equation}
The polynomials $p_{ \mathbf{k}}^r$ interpolate at the nodes
\begin{equation*}
\begin{split}
\mathcal{S}_{\mathbf{k}}^{r}&=\{x^{(i_1+k_1-r)}_{1}, \hdots, x^{(i_1+k_1)}_{1}\}\times\{x^{(i_2+k_2-r)}_{2}, \hdots, x^{(i_2+k_2)}_{2}\}\times\hdots\times\{x^{(i_n+k_n-r)}_{n}, \hdots, x^{(i_n+k_n)}_{n}\}=\prod_{j=1}^n\mathcal{S}_{k_j}^r.
\end{split}
\end{equation*}
The optimal weights are then replaced by non-linear ones using the following formula
\begin{equation}\label{omegand}
\beta^r_{\mathbf{k}}=\frac{\alpha^r_{\mathbf{k}}}{\sum_{\mathbf{l}\in\{0,\hdots,r-1\}^n}\alpha^r_{\mathbf{l}}},\,\,\text{with}\,\,
\alpha^r_{\mathbf{k}}=\frac{C^r_{\mathbf{k}}}{(\epsilon+I^r_{\mathbf{k}})^t},
\end{equation}
where $I^r_{\mathbf{k}}$ denote the smoothness indicators, which are required to satisfy the following conditions:
\begin{enumerate}[label={\bfseries P\arabic*}]
\item\label{P1sm} The order of the smoothness indicator in regions free of discontinuities is $h^2$, i.e.,
$$I^r_{\mathbf{k}}=O(h^2) \,\, \text{if}\,\, f \,\, \text{is smooth in } \,\, \mathcal{S}^r_{\mathbf{k}}.$$
\item\label{P2sm}The distance between two smoothness indicators in regions free of discontinuities is $h^{r+1}$, i.e.,
 let
  $\mathbf{k}=(k_1, \hdots, k_n)$, and $\mathbf{k}'=(k'_1, \hdots, k'_n)$
 be such that both $\mathcal{S}^r_{\mathbf{k}}$ and $\mathcal{S}^r_{\mathbf{k}'}$ are free of discontinuities, then
$$I^r_{\mathbf{k}}-I^r_{\mathbf{k}'}=O(h^{r+1}).$$
\item\label{P3sm} When the stencil $\mathcal{S}^r_{\mathbf{k}}$ is affected by a discontinuity,
 then:
$$I^r_{\mathbf{k}} \nrightarrow 0 \,\, \text{as}\,\, h\to 0.$$
\end{enumerate}


}

To construct the smoothness indicators in $\mathbb{R}^n$, we consider
\begin{equation} \label{set}
\mathcal{S}_k=\Omega_k \cap X = \{\mathbf{x}_{k_i} \in \Omega_k \cap X \colon i=1, \dots, N_k\},
\end{equation}
with $N_k>3$, and the linear least square problem \begin{equation} \label{problemLS}
p_k=\underset{p\in\Pi_1(\mathbb{R}^2)}{\arg \min} \sum_{i=1}^{N_k} (f(\mathbf{x}_{k_i})-p(\mathbf{x}_{k_i}))^2,
\end{equation} 
where \(\Pi_1(\mathbb{R}^2)\) denotes the space of all polynomials of degree at most one. Subsequently, we introduce the smoothness indicators as the average of the absolute errors, i.e.,
\begin{equation} \label{smoothind}
I_k=\frac{1}{N_k}\sum_{i=1}^{N_k} |f(\mathbf{x}_{k_i})-p_k(\mathbf{x}_{k_i})|,
\end{equation}
which satisfy the conditions {\color{black}(see, e.g., \cite{ABM,BBMZ})}:
\begin{itemize}
\item The order of a smoothness indicator in regions free of discontinuities is \(h^2\); that is, \(I_k = \mathcal{O}(h^2)\) whenever \(f\) is smooth within \(\Omega_k\).
\item $I_k \not\rightarrow 0 \text{ as } h \rightarrow 0$ when a discontinuity crosses $\Omega_k$.
\end{itemize}

The operator is:
\begin{equation} \label{INLPUM}
\mathcal{Q}^\text{{\color{black}DD}}_\text{PU}(f)(\mathbf{x})=\sum_{k=1}^M \mathcal{W}_k (\mathbf{x})p_k(\mathbf{x}),
\end{equation}
using
$$\mathcal{W}_k(\mathbf{x})=\frac{\alpha_k(\mathbf{x})}{\sum_{j=1}^M \alpha_j(\mathbf{x})}, \quad \text{with } \alpha_k(\mathbf{x}) = \gamma_k \varphi_k(\mathbf{x})= \frac{\varphi_k(\mathbf{x})}{(\epsilon+I_k)^t}, \quad {\color{black}\varphi_k(\mathbf{x})=\varphi\left(\mu_k\frac{\|\mathbf{x}-\mathbf{x_k}\|}{h}\right),} $$
where $0<\gamma_k=(\epsilon+I_k)^{-t}${\color{black}, $\mu_k$ are some positive constants ($\mu_k$ independent of $h$),}  and $\varphi$ are compactly supported functions. {\color{black}Note that the constant $\frac{\mu_k}{h}$ is chosen in order to get $\varphi_k$ with support in $\Omega_k$.}
{\color{black} Let us} look at some important properties of this method. We consider a function \( \zeta : \mathbb{R}^n \to \mathbb{R},\, \zeta \in \mathcal{C}^1 \), and the sets
\begin{equation} \label{setsOmega}
\Gamma = \{ \mathbf{x} \in \Omega : \zeta(\mathbf{x}) = 0 \}, \quad \Omega^+ = \{ \mathbf{x} \in \Omega : \zeta(\mathbf{x}) \geq 0 \}, \quad \Omega^- = \Omega \setminus \Omega^+
\end{equation} 
that is to say that $\Omega$ is divided into two subsets, $\Omega = \Omega^+ \cup \Omega^-$. Let us consider that the data under consideration is generated by a piecewise-defined function of the form
$$
f(\mathbf{x}) = \begin{cases}
f^+(\mathbf{x}), & \mathbf{x} \in \Omega^+, \\
f^-(\mathbf{x}), & \mathbf{x} \in \Omega^-,
\end{cases}
$$
where $f^\pm \in \mathcal{C}^k(\overline{\Omega^\pm}),\; k \in \mathbb{N}$, but

$$
\lim_{\mathbf{x} \to \mathbf{x}_0 \in \Gamma} f^-(\mathbf{x}) \neq f^+(\mathbf{x}_0).
$$ 
Accordingly, the set of scattered data points $X = \{ \mathbf{x}_i \}_{i=1}^N$ is naturally partitioned into two subsets, $X^\pm = X \cap \Omega^\pm$. In other words, we consider a domain in $\mathbb{R}^n$ that is divided into two parts separated by a smooth curve. The function we want to approximate is smooth in each part but has a jump discontinuity across the curve, so we split the data into two sets depending on which side of the discontinuity the points belong to. \textcolor{black}{This two-region setting is considered for simplicity of presentation. The proposed {\color{black} DDPU}--MLS method naturally extends to multiple or intersecting  discontinuities, since the non-linear weights locally adapt to the regularity of the data in each patch, in the same spirit as classical WENO schemes.}

We aim to demonstrate that, in smooth regions, the approximation preserves the expected order of convergence. As a consequence, it is observed that the method effectively suppresses the Gibbs phenomenon.





\begin{proposition}
Let \( \Omega \subset \mathbb{R}^n \) be an open and bounded domain such that \( \Omega \subseteq \bigcup_{k=1}^M \Omega_k \), where each \( \Omega_k \subset \mathbb{R}^n \) is open. Suppose a function \( \zeta : \mathbb{R}^n \to \mathbb{R}\) that divides our domain into two subsets, \( \Omega = \Omega^+ \cup \Omega^- \), where $\Omega^+$ and $\Omega^-$ are defined as in Eq. \eqref{setsOmega}. Let \( h > 0 \) denote the local fill distance. Given a discrete set \( X \subset \Omega \), define the local sample set $\mathcal{S}_k$ defined in Eq. \eqref{set}, with $N_k>3$. For each \( k \), consider the linear least-squares polynomial approximation $p_k$ defined in \eqref{problemLS}, and the local smoothness indicator is defined by Eq. \eqref{smoothind}. Let \( \varphi_k(\mathbf{x}) \) be compactly supported weight functions supported on \( \Omega_k \), and define the non-linear weights by
\[
\alpha_k(\mathbf{x}) = \frac{\varphi_k(\mathbf{x})}{(\epsilon + I_k)^t}, \quad \mathcal{W}_k(\mathbf{x}) = \frac{\alpha_k(\mathbf{x})}{\sum_{j=1}^M \alpha_j(\mathbf{x})},
\]
where \( \epsilon > 0 \) and \( t > 0 \) are fixed parameters.

Then, the following statements hold:
\begin{itemize}
    \item If \( f \in \mathcal{C}^{2}(\overline{\Omega_k}) \), then
    \[
    \|f - p_k\|_{\infty, \Omega_k} = O(h^{2}) \quad \text{as } h \to 0.
    \]
    \item If \( f \) has a discontinuity in \( \Omega_{k_0} \), then
    \[
    \|f - p_{k_0}\|_{\infty, \Omega_{k_0}} = O(1) \quad \text{as } h \to 0.
    \]
\end{itemize}

As a consequence, the normalized weights satisfy:
$$
\mathcal{W}_k(\mathbf{x}) = 
\begin{cases}
O(1), & \text{if } f \in \mathcal{C}^{2}(\overline{\Omega_k}), \\
O(h^{2t}), & \text{if } f \text{ is not smooth in } \overline{\Omega_k}.
\end{cases}
$$
\end{proposition}

\begin{proof}
Assume $f \in \mathcal{C}^{2}(\overline{\Omega_k})$. Then, by standard least-squares approximation theory over quasi-uniform samples with polynomials of degree 1, we have:
$$
|f(\mathbf{x}_{k_i}) - p_k(\mathbf{x}_{k_i})| = O(h^{2}), \quad \text{for all } i = 1, \dots, N_k,
$$
which implies
$$
I_k = \frac{1}{N_k} \sum_{i=1}^{N_k} |f(\mathbf{x}_{k_i}) - p_k(\mathbf{x}_{k_i})| = O(h^{2}).
$$
Hence,
$$
\alpha_k(\mathbf{x}) = \frac{\varphi_k(\mathbf{x})}{(\epsilon + O(h^{2}))^t} = O(h^{-2t}),
$$
and the normalized weight is
$$
\mathcal{W}_k(\mathbf{x}) = \frac{\alpha_k(\mathbf{x})}{\sum_{j=1}^M \alpha_j(\mathbf{x})} = \frac{O(h^{-2t})}{O(h^{-2t})} = O(1).
$$

Now suppose $f$ is not smooth in $\overline{\Omega_k}$ (i.e., it has a discontinuity or a singularity). Then, the least-squares error does not vanish as $h \to 0$, and
$$
|f(\mathbf{x}_{k_i}) - p_k(\mathbf{x}_{k_i})| = O(1),
$$
so that
$$
I_k = O(1), \quad \text{and therefore,} \quad \alpha_k(\mathbf{x}) = \frac{\varphi_k(\mathbf{x})}{(\epsilon + O(1))^t} = O(1).
$$
Comparing with the smooth case, where $\alpha_j(\mathbf{x}) = O(h^{-2t})$, we get
$$
\mathcal{W}_k(\mathbf{x}) = \frac{O(1)}{O(h^{-2t}) + O(1)} = O(h^{2t}).
$$ 

So we distinguish two regions in the domain: one where the function is smooth, and another where it is not. In the smooth regions, the approximation error \( f(\mathbf{x}) - p_k(\mathbf{x}) \) behaves like \( O(h^2) \) as \( h \to 0 \), and the weights \( \mathcal{W}_k(\mathbf{x}) \) remain bounded, that is, they are \( O(1) \). In contrast, in non-smooth regions, such as near a discontinuity, the approximation error does not decrease with \( h \); it stays of order \( O(1) \).

As we have seen in the previous calculations, in order to reach the expected accuracy, the weights in the non-smooth regions must be of order \( O(h^{2t}) \). This choice makes sure that the influence of the non-smooth regions on the total error becomes small when \( h \to 0 \), so that the final approximation keeps the desired order of accuracy.

\end{proof}

The smoothness of the new approximation operator depends only on the function $\varphi_k$ since $\alpha_k=\gamma_k \varphi_k$. We summarize this property in the next proposition.

\begin{proposition} \label{orderNL}
    Let $\nu\in \mathbb{N}$. Let \( \Omega \subset \mathbb{R}^n \) be an open and bounded domain such that \( \Omega \subseteq \bigcup_{k=1}^M \Omega_k \), where each \( \Omega_k \subset \mathbb{R}^n \) is open. If the weight functions $\varphi_k \in \mathcal{C}^{\nu}(\overline{\Omega_k})$ for all \( k = 1, \dots, M \), then the new approximation defined in Eq. \eqref{INLPUM} is $\mathcal{C}^{\nu}(\overline{\Omega})$.
\end{proposition}

\begin{proof}
    By Eq. \eqref{INLPUM}, since each \( \gamma_k \in \mathbb{R}^+ \) is a constant and \( \varphi_k \in \mathcal{C}^{\nu}(\overline{\Omega_k}) \), it follows that  $$\alpha_k(\mathbf{x}) = \gamma_k \varphi_k(\mathbf{x}) \in \mathcal{C}^{\nu}(\overline{\Omega_k}).$$ Moreover, the denominator $\sum_{j=1}^M \alpha_j(\mathbf{x})$ is a finite sum of \( \mathcal{C}^{\nu} \) functions, so it is also \( \mathcal{C}^{\nu} \), and strictly positive on \( \Omega \) due to the coverage assumption \( \Omega \subseteq \bigcup_{k=1}^M \Omega_k \) and the positivity of \( \varphi_k \) inside their supports. Therefore, the normalized weights
    $$\mathcal{W}_k(\mathbf{x}) = \frac{\alpha_k(\mathbf{x})}{\sum_{j=1}^M \alpha_j(\mathbf{x})} \in \mathcal{C}^{\nu}(\overline{\Omega}).$$

    Since each \( p_k \) is a polynomial function, it is infinitely differentiable and hence belongs to \( \mathcal{C}^\nu(\overline{\Omega_k}) \) for any \( \nu \). The product \( \mathcal{W}_k(\mathbf{x}) p_k(\mathbf{x}) \) is then \( \mathcal{C}^{\nu} \), and the finite sum
    \[
    \mathcal{Q}^{\mathrm{NL}}_{\mathrm{PU}}(f)(\mathbf{x}) = \sum_{k=1}^M \mathcal{W}_k(\mathbf{x}) p_k(\mathbf{x})
    \]
    is also \( \mathcal{C}^{\nu} \) on \( \overline{\Omega} \). Hence, the global {\color{black}data-dependent} PU-MLS approximation is \( \mathcal{C}^{\nu}(\overline{\Omega}) \).
\end{proof}

\begin{remark}
The {\color{black} DDPU}-MLS method constructed above effectively mitigates the Gibbs phenomenon near discontinuities. This is achieved through the non-linear weighting mechanism, which assigns significantly smaller weights to local approximants constructed over regions where the function lacks smoothness. As a result, the global approximation avoids the spurious oscillations typically observed in standard MLS or linear PU-MLS approaches when approximating functions with jumps or discontinuities. We will look at this fact in the next section on numerical experiments.
\end{remark}

\section{Numerical experiments} \label{num} 

In this section, we provide a series of numerical examples to confirm the results established in this work and to examine the properties of the proposed {\color{black}data-dependent} algorithm, {\color{black} DDPU}-MLS, in relation to the conventional PU-MLS method. In Subsection \ref{order}, we begin with an example designed to analyze the order of accuracy using a smooth test case, namely, the well-known Franke's function as given in Eq. \eqref{Franke}. The subsequent subsection focuses on the suppression of undesired oscillations near discontinuities. To this end, we revisit the example introduced in Sec. \ref{intro}, which serves to illustrate the effectiveness of the proposed algorithm for this class of functions. The numerical tests are conducted using both a regular grid and Halton sequences as sampling points. In all cases, the same parameter settings as in \cite{cavo} are employed; specifically, the number of subdomains along a single direction is given by $\lfloor \sqrt{N}/2 \rfloor$, with the radius chosen as $\delta=\sqrt{2/M}$. {\color{black} All numerical experiments are performed in two spatial dimensions, consistently with the multivariate formulation of the {\color{black} DDPU}--MLS method introduced in Section \ref{sec4}.}


\subsection{Order of accuracy} \label{order} 
We begin our analysis of the approximation order by considering the well-known Franke’s function, given in Eq.~\eqref{Franke}. The numerical tests employ two types of data distributions: a uniform grid $X_{N=2^l+1}=\{(i/2^l,j/2^l)\, \colon \, i,j=0,\dots, 2^l \}$ and a collection of $N=(2^l+1)^2$ scattered data points generated from the Halton sequence, as described in \cite{H}. The parameter $h_l$ denotes the corresponding fill distance, while the pointwise errors are defined as
$$e^l_j=|f({\color{black}\mathbf{z}}_j)-\mathcal{Q}^l({\color{black}\mathbf{z}}_j)|,$$ 
{\color{black} where $\{{\color{black}\mathbf{z}}_j\}_{j=1}^{N_{\mathrm{eval}}}$ denotes a uniform evaluation grid over $[0,1]^2$, with $N_{\mathrm{eval}} = 120^2$,} and $\mathcal{Q}^l$ stands for either $\mathcal{Q}_\text{PU}(f)$ or $\mathcal{Q}^\text{{\color{black}DD}}_\text{PU}(f)$ at each refinement level $l$. {\color{black} Here, the index $j$ labels the evaluation points in the two-dimensional domain, and $e_j^l$ denotes the pointwise error at the location $\mathbf{z}_j \in \mathbb{R}^2$.} Finally, the discrete maximum and root-mean-square errors, together with their associated convergence rates, are computed as
\begin{equation*} 
\begin{split}
\text{MAE}_l=\max_{j=1,\hdots,120} e^l_j,  \quad \text{RMSE}_l=\left(\frac{1}{120} \sum_{j=1}^{120} (e^l_j)^2\right)^{\frac{1}{2}}, \quad r_l^{\infty}=\frac{\log(\text{MAE}_{l-1}/\text{MAE}_{l})}{\log(h_{l-1}/h_{l})}, \quad r_l^{2}=\frac{\log(\text{RMSE}_{l-1}/\text{RMSE}_{l})}{\log(h_{l-1}/h_{l})}.
\end{split}
\end{equation*}

We use the acronyms PU-MLS$^p_\mathcal{H}$ and {\color{black} DDPU}-MLS$^p_\mathcal{H}$ when the operators shown in Eqs. \eqref{PUM} and \eqref{INLPUM} are used, with $p$ representing the degree of the polynomials and $\mathcal{H}=$ G, W2, W4 indicating that $w$ is Gaussian, Wendland $\mathcal{C}^2$ or Wendland $\mathcal{C}^4$ function, as shown in Table \ref{tabla1nucleos}. {\color{black} We apply these functions in the form $w_i(\mathbf{x})=w(\gamma_k \|\mathbf{x}-\mathbf{x}_i\|)$, where the shape parameter is defined as $\gamma_k = \frac{\mu_k}{h}$ for each refinement level $k$. We take $\gamma_k=\lfloor \sqrt{N}/2 \rfloor / \sqrt{2}$ for W2 and W4, and $\gamma_k=\sqrt{2} \,\lfloor \sqrt{N}/2 \rfloor$ for G.} In the Gaussian case, we only consider the values $\mathbf{x}_i \in X$ such that $w_i(\mathbf{x})>10^{-10}$.


\begin{table}[h!]
\begin{center}
\begin{tabular}{lrrrrrrrrrrr}
& \multicolumn{2}{c}{PU-MLS$^2_{\text{W2}}$} & &   \multicolumn{2}{c}{{\color{black} DDPU}-MLS$^2_{\text{W2}}$}& &\multicolumn{2}{c}{PU-MLS$^2_{\text{W2}}$} & &   \multicolumn{2}{c}{{\color{black} DDPU}-MLS$^2_{\text{W2}}$}\\ \cline{1-3} \cline{5-6} \cline{8-9} \cline{11-12}
$l$ & $\text{MAE}_l$ & $r_l^{\infty}$ & &$\text{MAE}_l$ & $r_l^{\infty}$ & & $\text{RMSE}_l$ & $r_l^{2}$ & &$\text{RMSE}_l$ & $r_l^{2}$       \\
\hline
$4$  &                     1.0660e-02 &       & & 2.8853e-02 &      &&         1.2781e-03 &       &&  3.7739e-03 &          \\
$5$  &                     8.9460e-04 &  3.5748 & & 2.6184e-03 &  3.4620&&         1.0177e-04 &  3.6506 &&  4.2879e-04 &  3.1377    \\
$6$  &                     6.7838e-05 &  3.7211 & & 4.4858e-04 &  2.5453&&         7.6350e-06 &  3.7365 &&  5.2508e-05 &  3.0297    \\
$7$ &                     5.3291e-06 &  3.6701 & & 4.6282e-05 &  3.2768&&         6.6164e-07 &  3.5285 &&  6.3381e-06 &  3.0504    \\
\hline
\hline
& \multicolumn{2}{c}{PU-MLS$^2_{\text{W4}}$} & &   \multicolumn{2}{c}{{\color{black} DDPU}-MLS$^2_{\text{W4}}$}& &\multicolumn{2}{c}{PU-MLS$^2_{\text{W4}}$} & &   \multicolumn{2}{c}{{\color{black} DDPU}-MLS$^2_{\text{W4}}$}\\ \cline{1-3} \cline{5-6} \cline{8-9} \cline{11-12}
$l$ & $\text{MAE}_l$ & $r_l^{\infty}$ & &$\text{MAE}_l$ & $r_l^{\infty}$ & & $\text{RMSE}_l$ & $r_l^{2}$ & &$\text{RMSE}_l$ & $r_l^{2}$       \\
\hline
$4$  &                  8.4714e-03 &        && 3.0652e-02 &       &&   9.8372e-04  &     & &  3.4076e-03 &        \\
$5$  &                  7.4136e-04 &  3.5144  && 2.4897e-03 &  3.6219 &&   7.8304e-05  & 3.6511& &  3.7181e-04 &  3.1961 \\
$6$  &                  5.8730e-05 &  3.6580  && 3.9069e-04 &  2.6719 &&   6.6490e-06  & 3.5579& &  4.4703e-05 &  3.0561 \\
$7$ &                  4.2374e-06 &  3.7928  && 3.6824e-05 &  3.4073 &&   5.3169e-07  & 3.6445& &  5.3467e-06 &  3.0636 \\
\hline
\hline
& \multicolumn{2}{c}{PU-MLS$^2_{\text{G}}$} & &   \multicolumn{2}{c}{{\color{black} DDPU}-MLS$^2_{\text{G}}$}& &\multicolumn{2}{c}{PU-MLS$^2_{\text{G}}$} & &
\multicolumn{2}{c}{{\color{black} DDPU}-MLS$^2_{\text{G}}$}\\ \cline{1-3} \cline{5-6} \cline{8-9} \cline{11-12}
$l$ & $\text{MAE}_l$ & $r_l^{\infty}$ & &$\text{MAE}_l$ & $r_l^{\infty}$ & & $\text{RMSE}_l$ & $r_l^{2}$ & &$\text{RMSE}_l$ & $r_l^{2}$       \\ \hline

$4$  & 1.8314e-02 &   & & 5.4331e-02 &   && 2.4049e-03 &    && 8.0603e-03 &     \\
$5$  & 1.7815e-03 & 3.3617& & 9.7512e-03 & 2.4781&& 2.0424e-04 & 3.5577&& 1.0028e-03 & 3.0068  \\
$6$  & 1.3656e-04 & 3.7055& & 1.0640e-03 & 3.1961&& 1.5270e-05 & 3.7415&& 1.2148e-04 & 3.0453  \\
$7$  & 2.1561e-05 & 2.6631& & 9.9053e-05 & 3.4252&& 2.5381e-06 & 2.5889&& 1.3977e-05 & 3.1196  \\

\hline
\end{tabular}
\end{center}
\caption{\scriptsize{Errors and rates using linear and {\color{black} data-dependent} PU-MLS methods (degree 2) for Franke's test function evaluated at grid points.}}\label{exp11}
\end{table}

\begin{table}[h!]
\begin{center}
\begin{tabular}{lrrrrrrrrrrr}
& \multicolumn{2}{c}{PU-MLS$^3_{\text{W2}}$} & &   \multicolumn{2}{c}{{\color{black} DDPU}-MLS$^3_{\text{W2}}$}& &\multicolumn{2}{c}{PU-MLS$^3_{\text{W2}}$} & &   \multicolumn{2}{c}{{\color{black} DDPU}-MLS$^3_{\text{W2}}$}\\ \cline{1-3} \cline{5-6} \cline{8-9} \cline{11-12}
$l$ & $\text{MAE}_l$ & $r_l^{\infty}$ & &$\text{MAE}_l$ & $r_l^{\infty}$ & & $\text{RMSE}_l$ & $r_l^{2}$ & &$\text{RMSE}_l$ & $r_l^{2}$       \\
\hline
$2$  & 2.6635e-01 &       & & 1.2665e+00 &      && 7.5969e-02 &       && 6.9229e-02 &          \\
$3$  & 5.8942e-02 & 2.1759 & & 6.4884e-02 & 4.2868 && 1.0215e-02 & 2.8947 && 1.2481e-02 & 2.4717 \\
$4$  & 1.0321e-02 & 2.5138 & & 1.8325e-02 & 1.8241 && 1.1200e-03 & 3.1891 && 1.5477e-03 & 3.0115 \\
$5$  & 8.1815e-04 & 3.6570 & & 1.2601e-03 & 3.8621 && 8.2028e-05 & 3.7712 && 9.9590e-05 & 3.9580 \\
\hline
\hline
& \multicolumn{2}{c}{PU-MLS$^3_{\text{W4}}$} & &   \multicolumn{2}{c}{{\color{black} DDPU}-MLS$^3_{\text{W4}}$}& &\multicolumn{2}{c}{PU-MLS$^3_{\text{W4}}$} & &   \multicolumn{2}{c}{{\color{black} DDPU}-MLS$^3_{\text{W4}}$}\\ \cline{1-3} \cline{5-6} \cline{8-9} \cline{11-12}
$l$ & $\text{MAE}_l$ & $r_l^{\infty}$ & &$\text{MAE}_l$ & $r_l^{\infty}$ & & $\text{RMSE}_l$ & $r_l^{2}$ & &$\text{RMSE}_l$ & $r_l^{2}$       \\
\hline
$2$  & 2.8200e-01 &        && 1.2665e+00 &       && 7.8948e-02 &     & & 7.0630e-02 &        \\
$3$  & 5.7881e-02 & 2.2845 && 6.2876e-02 & 4.3322 && 9.4199e-03 & 3.0671 & & 1.1173e-02 & 2.6603 \\
$4$  & 8.4535e-03 & 2.7755 && 1.7929e-02 & 1.8102 && 9.1034e-04 & 3.3712 & & 1.4273e-03 & 2.9686 \\
$5$  & 7.0882e-04 & 3.5760 && 8.5953e-04 & 4.3826 && 6.6973e-05 & 3.7648 & & 8.4005e-05 & 4.0867 \\
\hline
\hline
& \multicolumn{2}{c}{PU-MLS$^3_{\text{G}}$} & &   \multicolumn{2}{c}{{\color{black} DDPU}-MLS$^3_{\text{G}}$}& &\multicolumn{2}{c}{PU-MLS$^3_{\text{G}}$} & &
\multicolumn{2}{c}{{\color{black} DDPU}-MLS$^3_{\text{G}}$}\\ \cline{1-3} \cline{5-6} \cline{8-9} \cline{11-12}
$l$ & $\text{MAE}_l$ & $r_l^{\infty}$ & &$\text{MAE}_l$ & $r_l^{\infty}$ & & $\text{RMSE}_l$ & $r_l^{2}$ & &$\text{RMSE}_l$ & $r_l^{2}$       \\ \hline
$2$ & 2.4911e-01 &      & & 2.5331e+00 &      && 7.0075e-02 &      && 1.2272e-01 &       \\
$3$ & 2.4548e-01 & 0.0212 & & 8.1384e-02 & 4.9600  && 4.0979e-02 & 0.7740 && 9.8184e-03 & 3.6437 \\
$4$ & 1.3354e-02 & 4.2003 & &1.4330e-02 & 2.5057   && 1.5689e-03 & 4.7071 && 6.8683e-04 & 3.8375 \\
$5$ & 1.2850e-03 & 3.3775 & & 5.1022e-04 & 4.8117 && 1.3171e-04 & 3.5743 && 3.3476e-05 & 4.3587  \\
\hline
\end{tabular}
\end{center}
\caption{\scriptsize{Errors and rates using linear and {\color{black} data-dependent} PU-MLS methods (degree 3) for Franke's test function evaluated at grid points.}}\label{exp12}
\end{table}

Tables \ref{exp11} and \ref{exp12} confirm that, on uniform grids, the numerical order of accuracy matches the theoretical predictions: close to 3 for quadratic reconstruction ($p=2$) and close to 4 for cubic reconstruction ($p=3$). In both cases, PU-MLS and {\color{black} DDPU}-MLS display comparable convergence rates. The errors decrease consistently with refinement, with the {\color{black} DDPU}-MLS method yielding slightly smaller values in general.




Now, we repeat the experiment on non-uniform grids by using Halton's points within the interval $[0,1]^2$. {\color{black} As before, we use the same values of the shape parameter $\gamma_k$.}

\begin{table}[h!]
\begin{center}
\begin{tabular}{lrrrrrrrrrrr}
& \multicolumn{2}{c}{PU-MLS$^2_{\text{W2}}$} & &   \multicolumn{2}{c}{{\color{black} DDPU}-MLS$^2_{\text{W2}}$}& &\multicolumn{2}{c}{PU-MLS$^2_{\text{W2}}$} & &   \multicolumn{2}{c}{{\color{black} DDPU}-MLS$^2_{\text{W2}}$}\\ \cline{1-3} \cline{5-6} \cline{8-9} \cline{11-12}
$l$ & $\text{MAE}_l$ & $r_l^{\infty}$ & &$\text{MAE}_l$ & $r_l^{\infty}$ & & $\text{RMSE}_l$ & $r_l^{2}$ & &$\text{RMSE}_l$ & $r_l^{2}$       \\
\hline
$2$  & 1.2181e+00 &       & & 1.2181e+00 &       && 1.0833e-01 &       && 1.4901e-01 &        \\
$3$  & 2.6721e-01 & 2.1886& & 2.6721e-01 & 2.1886&& 1.6203e-02 & 2.7411&& 2.6141e-02 & 2.5110 \\
$4$  & 6.4666e-02 & 2.0469& & 6.4812e-02 & 2.0436&& 2.7382e-03 & 2.5650&& 4.2469e-03 & 2.6218 \\
$5$  & 4.6229e-03 & 3.8061& & 3.7808e-03 & 4.0995&& 2.6968e-04 & 3.3439&& 5.3960e-04 & 2.9765 \\
\hline
\hline
& \multicolumn{2}{c}{PU-MLS$^2_{\text{W4}}$} & &   \multicolumn{2}{c}{{\color{black} DDPU}-MLS$^2_{\text{W4}}$}& &\multicolumn{2}{c}{PU-MLS$^2_{\text{W4}}$} & &   \multicolumn{2}{c}{{\color{black} DDPU}-MLS$^2_{\text{W4}}$}\\ \cline{1-3} \cline{5-6} \cline{8-9} \cline{11-12}
$l$ & $\text{MAE}_l$ & $r_l^{\infty}$ & &$\text{MAE}_l$ & $r_l^{\infty}$ & & $\text{RMSE}_l$ & $r_l^{2}$ & &$\text{RMSE}_l$ & $r_l^{2}$       \\
\hline
$2$  & 1.5395e+00 &       & & 1.5395e+00 &       && 1.2517e-01 &       && 1.8510e-01 &        \\
$3$  & 2.6657e-01 & 2.5299& & 2.6657e-01 & 2.5299&& 1.5564e-02 & 3.0075&& 2.3416e-02 & 2.9827 \\
$4$  & 6.4795e-02 & 2.0406& & 6.4812e-02 & 2.0402&& 2.6643e-03 & 2.5464&& 3.9539e-03 & 2.5661 \\
$5$  & 3.3082e-03 & 4.2917& & 3.5490e-03 & 4.1908&& 2.6460e-04 & 3.3318&& 4.7547e-04 & 3.0559 \\
\hline
\hline
& \multicolumn{2}{c}{PU-MLS$^2_{\text{G}}$} & &   \multicolumn{2}{c}{{\color{black} DDPU}-MLS$^2_{\text{G}}$}& &\multicolumn{2}{c}{PU-MLS$^2_{\text{G}}$} & & \multicolumn{2}{c}{{\color{black} DDPU}-MLS$^2_{\text{G}}$}\\ \cline{1-3} \cline{5-6} \cline{8-9} \cline{11-12}
$l$ & $\text{MAE}_l$ & $r_l^{\infty}$ & &$\text{MAE}_l$ & $r_l^{\infty}$ & & $\text{RMSE}_l$ & $r_l^{2}$ & &$\text{RMSE}_l$ & $r_l^{2}$       \\ \hline
$2$  & 5.2726e-01 &   & & 5.2726e-01 &   && 1.0378e-01 &   && 8.1783e-02 &     \\
$3$  & 8.1774e-02 & 2.6888& & 1.2125e-01 & 2.1206&& 1.8985e-02 & 2.4507&& 2.8204e-02 & 1.5359  \\
$4$  & 1.5458e-02 & 2.4033& & 4.4310e-02 & 1.4522&& 2.8313e-03 & 2.7453&& 7.3400e-03 & 1.9420  \\
$5$  & 3.3042e-03 & 2.2260& & 9.9623e-03 & 2.1531&& 3.5984e-04 & 2.9760&& 1.0553e-03 & 2.7982  \\
\hline
\end{tabular}
\end{center}
\caption{\scriptsize{Errors and rates using linear and {\color{black} data-dependent} PU-MLS methods (degree 2) for the provided test on Halton's points.}}\label{exp23}
\end{table}

\begin{table}[h!]
\begin{center}
\begin{tabular}{lrrrrrrrrrrr}
& \multicolumn{2}{c}{PU-MLS$^3_{\text{W2}}$} & &   \multicolumn{2}{c}{{\color{black} DDPU}-MLS$^3_{\text{W2}}$}& &\multicolumn{2}{c}{PU-MLS$^3_{\text{W2}}$} & &   \multicolumn{2}{c}{{\color{black} DDPU}-MLS$^3_{\text{W2}}$}\\ \cline{1-3} \cline{5-6} \cline{8-9} \cline{11-12}
$l$ & $\text{MAE}_l$ & $r_l^{\infty}$ & &$\text{MAE}_l$ & $r_l^{\infty}$ & & $\text{RMSE}_l$ & $r_l^{2}$ & &$\text{RMSE}_l$ & $r_l^{2}$       \\
\hline
$2$  & 1.9404e+00 &       & & 1.9404e+00 &       && 1.7834e-01 &       && 1.9074e-01 &        \\
$3$  & 2.6038e-01 & 2.8977& & 2.6038e-01 & 2.8977&& 1.6501e-02 & 3.4340&& 2.1348e-02 & 3.1594 \\
$4$  & 5.8785e-02 & 2.1471& & 5.8897e-02 & 2.1443&& 2.2947e-03 & 2.8462&& 2.7146e-03 & 2.9753 \\
$5$  & 3.6878e-03 & 3.9946& & 2.5759e-03 & 4.5150&& 1.3960e-04 & 4.0389&& 1.5111e-04 & 4.1671 \\
\hline
\hline
& \multicolumn{2}{c}{PU-MLS$^3_{\text{W4}}$} & &   \multicolumn{2}{c}{{\color{black} DDPU}-MLS$^3_{\text{W4}}$}& &\multicolumn{2}{c}{PU-MLS$^3_{\text{W4}}$} & &   \multicolumn{2}{c}{{\color{black} DDPU}-MLS$^3_{\text{W4}}$}\\ \cline{1-3} \cline{5-6} \cline{8-9} \cline{11-12}
$l$ & $\text{MAE}_l$ & $r_l^{\infty}$ & &$\text{MAE}_l$ & $r_l^{\infty}$ & & $\text{RMSE}_l$ & $r_l^{2}$ & &$\text{RMSE}_l$ & $r_l^{2}$       \\
\hline
$2$  & 1.9421e+00 &       & & 1.9421e+00 &       && 1.8229e-01 &       && 1.8919e-01 &        \\
$3$  & 3.2329e-01 & 2.5867& & 2.6038e-01 & 2.8989&& 1.7048e-02 & 3.4186&& 2.0985e-02 & 3.1725 \\
$4$  & 5.8884e-02 & 2.4569& & 5.8897e-02 & 2.1443&& 2.1875e-03 & 2.9623&& 2.4949e-03 & 3.0723 \\
$5$  & 4.3867e-03 & 3.7467& & 2.5759e-03 & 4.5150&& 1.4931e-04 & 3.8729&& 1.4366e-04 & 4.1182 \\
\hline
\hline
& \multicolumn{2}{c}{PU-MLS$^3_{\text{G}}$} & &   \multicolumn{2}{c}{{\color{black} DDPU}-MLS$^3_{\text{G}}$}& &\multicolumn{2}{c}{PU-MLS$^3_{\text{G}}$} & & \multicolumn{2}{c}{{\color{black} DDPU}-MLS$^3_{\text{G}}$}\\ \cline{1-3} \cline{5-6} \cline{8-9} \cline{11-12}
$l$ & $\text{MAE}_l$ & $r_l^{\infty}$ & &$\text{MAE}_l$ & $r_l^{\infty}$ & & $\text{RMSE}_l$ & $r_l^{2}$ & &$\text{RMSE}_l$ & $r_l^{2}$       \\ \hline
$2$  & 5.4493e+00 &   & & 5.4463e+00 &   && 7.5698e-01 &   && 7.7851e-01 &     \\
$3$  & 2.9662e-01 & 4.1994& & 3.6358e+00 & 0.5830&& 4.4877e-02 & 4.0762&& 1.8134e-01 & 2.1020  \\
$4$  & 1.4680e-02 & 4.3367& & 8.0028e-02 & 5.5056&& 1.8423e-03 & 4.6064&& 9.2098e-03 & 4.2994  \\
$5$  & 2.1421e-03 & 2.7768& & 1.4321e-02 & 2.4824&& 1.5831e-04 & 3.5407&& 6.7622e-04 & 3.7676  \\
\hline
\end{tabular}
\end{center}
\caption{\scriptsize{Errors and rates using linear and {\color{black} data-dependent} PU-MLS methods (degree 3) for the provided test on Halton's points.}}\label{exp24}
\end{table}

Again, when the same experiments are repeated on non-uniform grids using Halton's points in $[0,1]^2$ (see Tables \ref{exp23} and \ref{exp24}), the observed orders of accuracy remain consistent with the theoretical estimates (see Th.~\ref{orderL} and Prop.~\ref{orderNL}). The convergence behavior of the two algorithms is analogous, while the non-linear variant produces marginally lower errors.

In summary, for both $p=2$ and $p=3$, and independently of whether the data are distributed on a uniform grid or on Halton's points, the linear and {\color{black} data-dependent} approaches exhibit similar performance in terms of accuracy and convergence rate. Nevertheless, the {\color{black} DDPU}-MLS method generally attains slightly better accuracy, even when the target function is smooth. 

{\color{black} It is important to emphasize that the goal of the {\color{black} DDPU}--MLS formulation is not to improve the approximation accuracy for smooth functions. In such cases, the linear PU--MLS method is already optimal, and the non-linear weights are designed to recover the same asymptotic behavior. The results reported in Tables~\ref{exp11} -- \ref{exp24} confirm that the proposed {\color{black} data-dependent} approach preserves the expected convergence rates and accuracy in smooth regions, without introducing any degradation.}

In the next subsection, the performance of the algorithms is examined in the presence of data exhibiting strong gradients or discontinuities.



\subsection{Avoiding oscillations close to the discontinuities} 
{\color{black} While the numerical results for smooth test functions show comparable behavior for PU--MLS and {\color{black} DDPU}--MLS, the differences between the two methods become substantial when discontinuities are present.} In this subsection, we perform several experiments starting with the piecewise smooth function ${\color{black} f_2}$, Eq. \eqref{frank}, using Eq. \eqref{Franke} defined in Section \ref{intro}.

Our goal is to compare the linear and {\color{black} data-dependent} PU-MLS methods. {\color{black} We use a uniform evaluation grid consisting of $N_{\mathrm{eval}} = 120^2$ points in the square $[0,1]^2$} and employ the Wendland $\mathcal{C}^2$ and $\mathcal{C}^4$ functions{\color{black}, using the shape parameters $\gamma_k$ defined in the previous subsection.} The results, together with the approximation errors between the original function and the approximated one for each method, are presented in Figures \ref{example1W2} and \ref{example1W4}. We observe that the {\color{black} data-dependent} method exhibits appropriate behavior near the discontinuities, whereas the linear method generates noticeable oscillations. When the figure is rotated (second column of the figures), these differences become even more pronounced. The error plots in the third column support this observation: the linear PU-MLS method produces a broader region of elevated error near the discontinuity, indicating smearing. In contrast, the {\color{black} DDPU}-MLS method yields a more confined error distribution, preserving the discontinuous structure more effectively. This observation will be further corroborated with additional numerical examples. 

 In order to analyze the behaviour when the discontinuity is more pronounced, we choose the function \eqref{Trig} (used in \cite{ALRY}), first using the Wendland $\mathcal{C}^2$ function (Figure \ref{example2W2}) and after using the Wendland $\mathcal{C}^4$ function (Figure \ref{example2W4}). 

 \begin{equation}\label{Trig}
g(x,y) = 
\begin{cases}
\sin(xy), & (x - 0.5)^2 + (y - 0.5)^2 \geq 0.25^2, \\
\cos(xy), & (x - 0.5)^2 + (y - 0.5)^2 < 0.25^2.
\end{cases}
\end{equation}

Some undesirable effects and smearing of the discontinuities appear when linear methods are used, which are reduced when the {\color{black} data-dependent} algorithms are employed. Moreover, as the discontinuity becomes more pronounced, these differences are even more apparent compared to the previous example. The error plots confirm this behavior: the linear method produces a wider region of elevated error around the discontinuity, while the non-linear method maintains a more localized and accurate reconstruction.



 Now, we study function \eqref{Trig2} to confirm that the constructed method avoids the Gibbs phenomenon, so that oscillations near the discontinuity and image diffusion are reduced. We can see the results obtained in Figures \ref{example3W2} and \ref{example3W4}, using respectively the radial basis functions Wendland $\mathcal{C}^2$ and $\mathcal{C}^4$ functions.

 \begin{equation} \label{Trig2}
 h(x,y) = 
\begin{cases}
y\sin(x)+y\cos(x), & (x - 0.5)^2 + (y - 0.5)^2 \geq 0.25^2, \\
\exp(xy)+1, & (x - 0.5)^2 + (y - 0.5)^2 < 0.25^2.
\end{cases}    
 \end{equation}

 As observed in the previous examples, the linear PU-MLS method produces oscillations close to the discontinuity when applied to the function \eqref{Trig2}. In contrast, the {\color{black} data-dependent} formulation effectively suppresses these artifacts, yielding a more accurate and stable approximation across the two regions. The error remains more localized in the non-linear method, confirming its ability to preserve discontinuities and reduce oscillations.




 Finally, we test \eqref{Trig3}, which includes both positive and negative jumps across the discontinuity.

  \begin{equation} \label{Trig3}
 j(x,y) = 
\begin{cases}
-(x+y+1)\cos(4x)+\sin(4(x+y)), & (x - 0.5)^2 + (y - 0.5)^2 \geq 0.1, \\
\exp(-10((x-0.5)^2+(y-0.5)^2)), & (x - 0.5)^2 + (y - 0.5)^2 < 0.1.
\end{cases}    
 \end{equation}



Figures \ref{example4W2} and \ref{example4W4} illustrate the comparison between the linear and non-linear methods. As in the previous cases, the {\color{black} data-dependent} approach consistently reduces error near discontinuities, offering a more stable and accurate approximation. These results confirm that undesired phenomena such as oscillations and smearing are effectively mitigated when the {\color{black} DDPU}-MLS method is applied.


\section{Conclusions} \label{conc}

In this work, we have proposed a {\color{black} data-dependent} enhancement of the Partition of Unity Moving Least Squares method ({\color{black} DDPU}-MLS). To the best of our knowledge, this is the first attempt in the literature to incorporate non-linear strategies within the PU-MLS framework. The motivation behind this development is the observation that the classical PUMLS scheme, although highly accurate in smooth regions, may lead to oscillatory behavior or loss of stability under demanding conditions. The {\color{black} DDPU}-MLS formulation introduces a modified weighting strategy that adapts locally, with the goal of reinforcing stability while preserving approximation quality.

The numerical experiments first confirm that, for smooth functions, the observed order of accuracy aligns with the theoretical predictions, with both PU-MLS and {\color{black} DDPU}-MLS achieving the expected convergence rates. More importantly, for functions featuring pronounced discontinuities, the {\color{black} data-dependent} method effectively suppresses the oscillations near the jumps that are present when the linear PU-MLS is applied. This indicates that the {\color{black} DDPU}-MLS approach successfully mitigates the Gibbs phenomenon while maintaining the approximation quality in smooth regions.

Overall, the results indicate that {\color{black} DDPU}-MLS is a reliable alternative to the standard PUMLS method. While both approaches display almost identical convergence properties in smooth test cases, the non-linear mechanism provides additional robustness that can be advantageous in more complex scenarios. Future research will focus on assessing the behavior of {\color{black} DDPU}-MLS in problems involving irregular data distributions or discontinuities, and on further refining the adaptive weighting strategy to reduce the need for manual parameter adjustment.


\section*{Conflict of Interest} 
The authors state that there are no conflicts of interest associated with this work.

\section*{Declaration of the Use of Generative AI and AI-Assisted Tools in the Writing Process} During the preparation of this work, the authors utilized Microsoft Copilot solely to assist with spelling and grammatical corrections. 
All content was subsequently reviewed and edited by the authors, who take full responsibility for the accuracy and integrity of the final publication.

\bibliographystyle{these}

\newpage
\begin{figure}[H]
\centering
    \begin{tabular}{ccc}
\multicolumn{3}{c}{\scriptsize{PU-MLS}}\\
         \includegraphics[width=0.3\hsize]{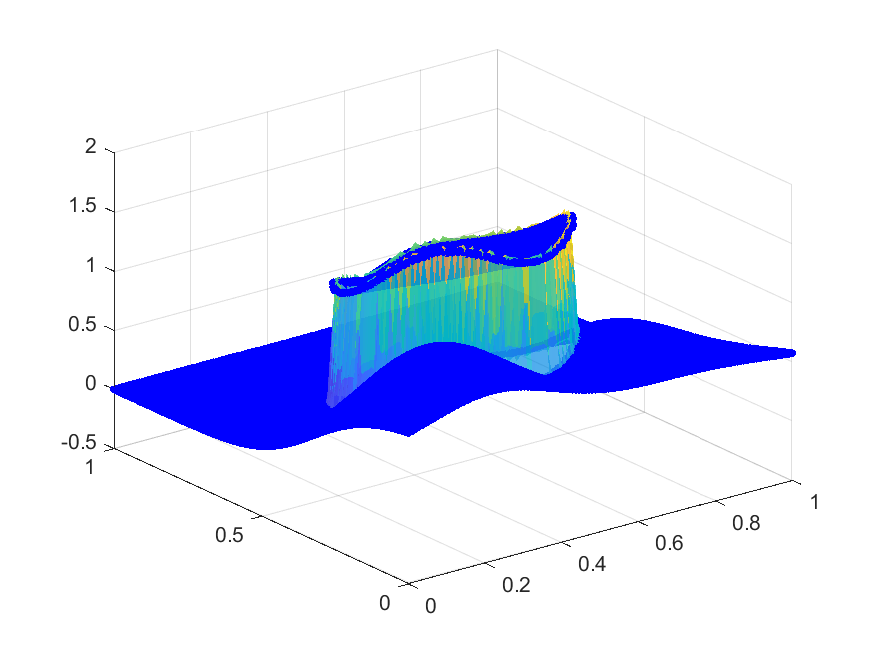} & \includegraphics[width=0.3\hsize]{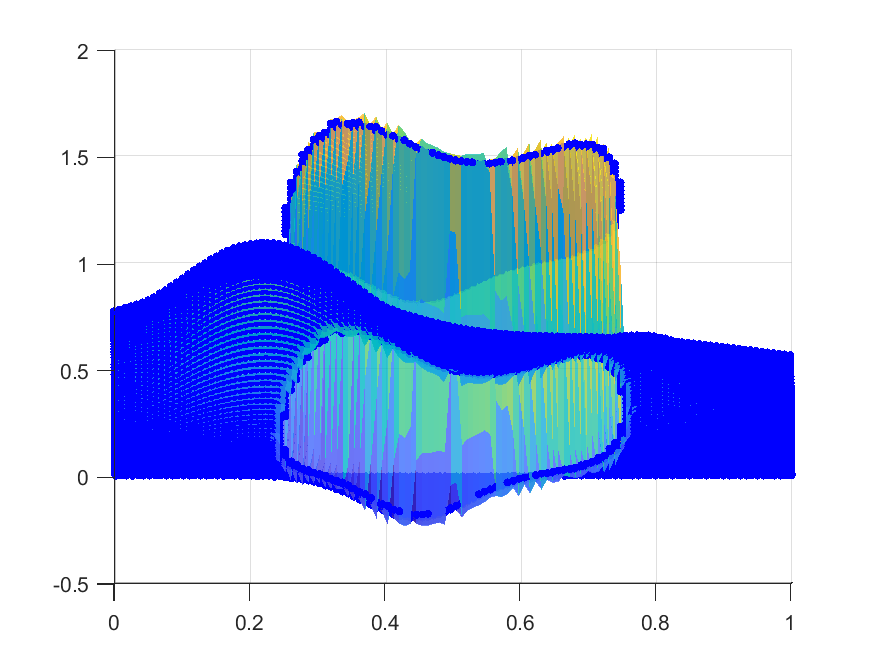} & \includegraphics[width=0.3\hsize]{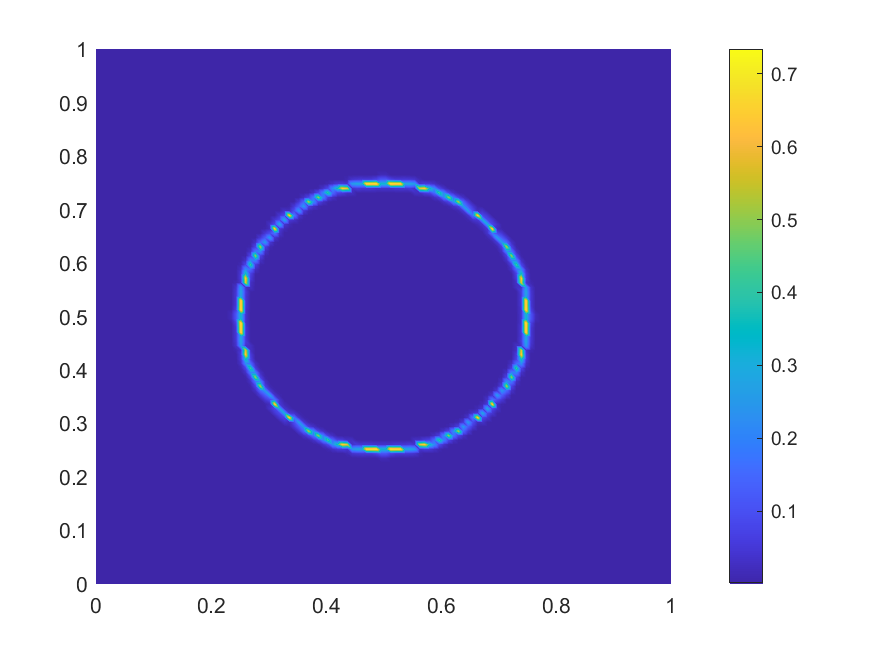}\\
\multicolumn{3}{c}{\scriptsize{{\color{black} DDPU}-MLS}}\\
         \includegraphics[width=0.3\hsize]{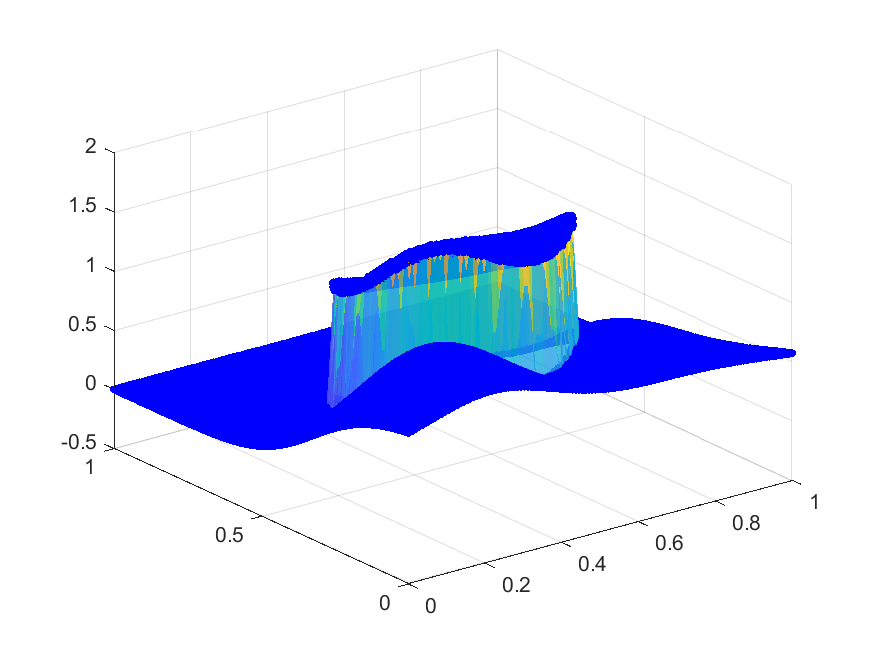} & \includegraphics[width=0.3\hsize]{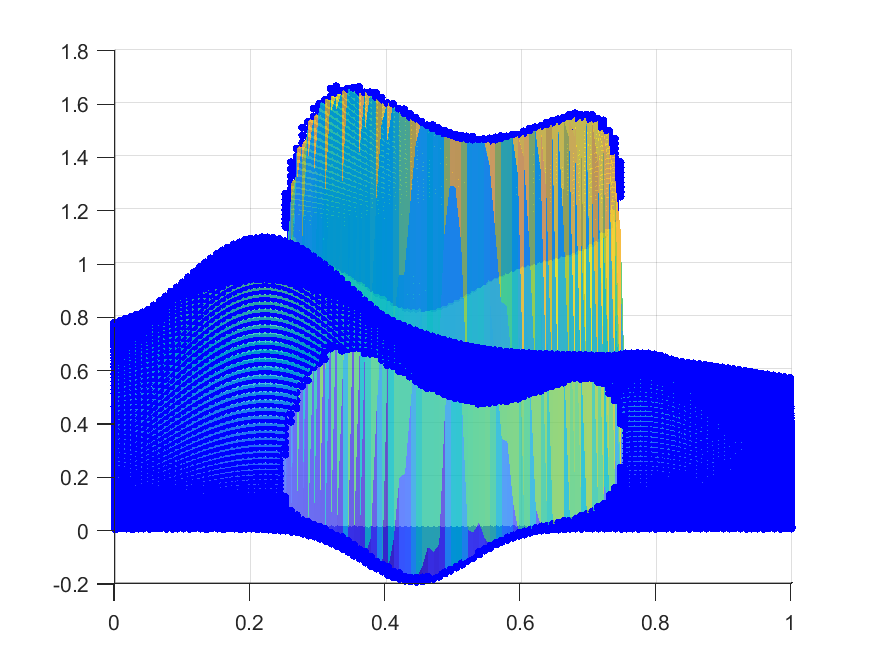}  & \includegraphics[width=0.3\hsize]{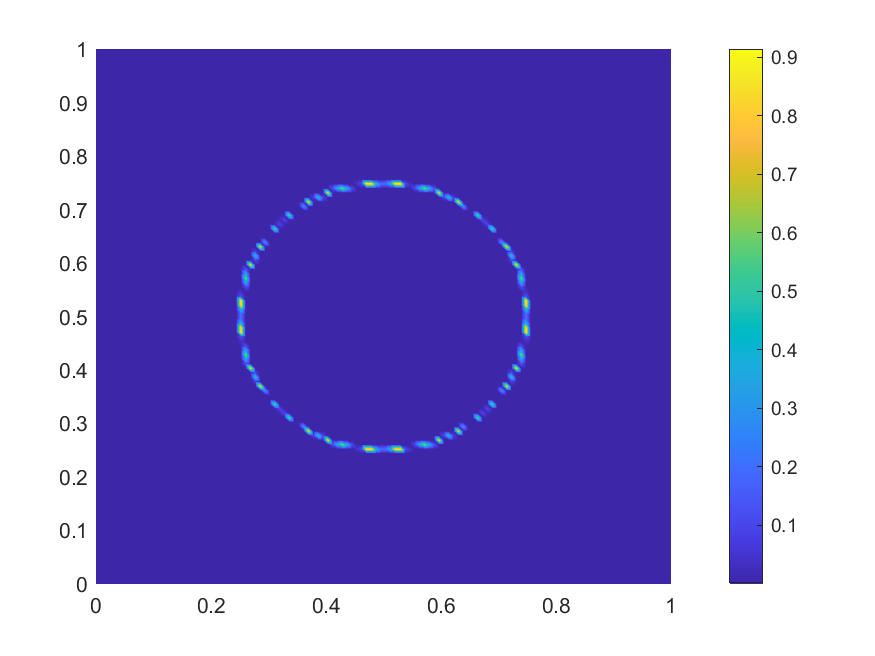}\\
 \end{tabular}
 \caption{\scriptsize{Approximation of the function {\color{black} $f_2$}, Eq. \eqref{frank}, using PU-MLS and {\color{black} DDPU}-MLS with the Wendland $\mathcal{C}^2$ function. The second column shows rotated views of the plots in the first column. The third column shows the errors between the original function and its approximation.}}
 \label{example1W2}
 \end{figure}

 \begin{figure}[H]
\centering
    \begin{tabular}{ccc}
\multicolumn{3}{c}{\scriptsize{PU-MLS}}\\
         \includegraphics[width=0.3\hsize]{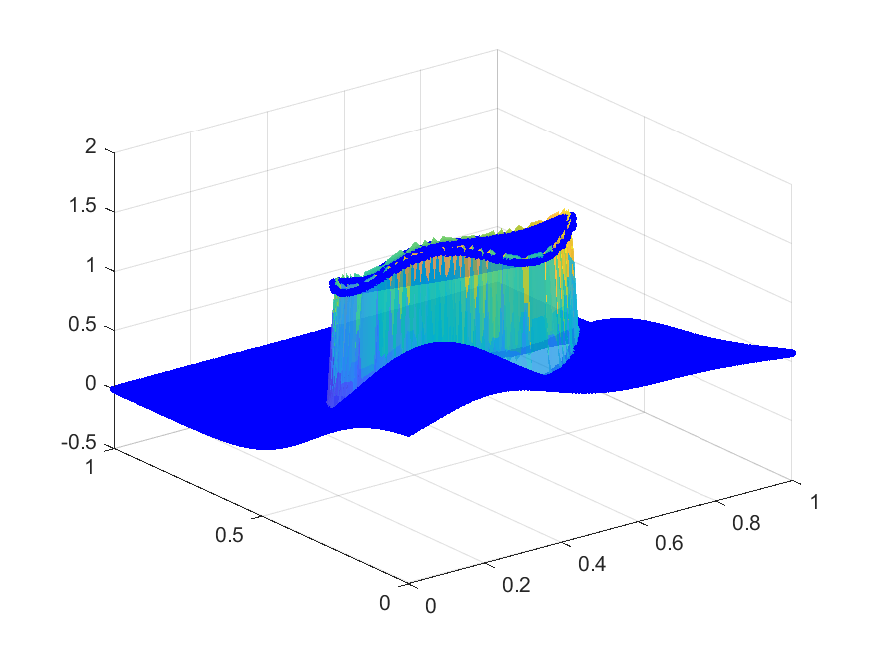} & \includegraphics[width=0.3\hsize]{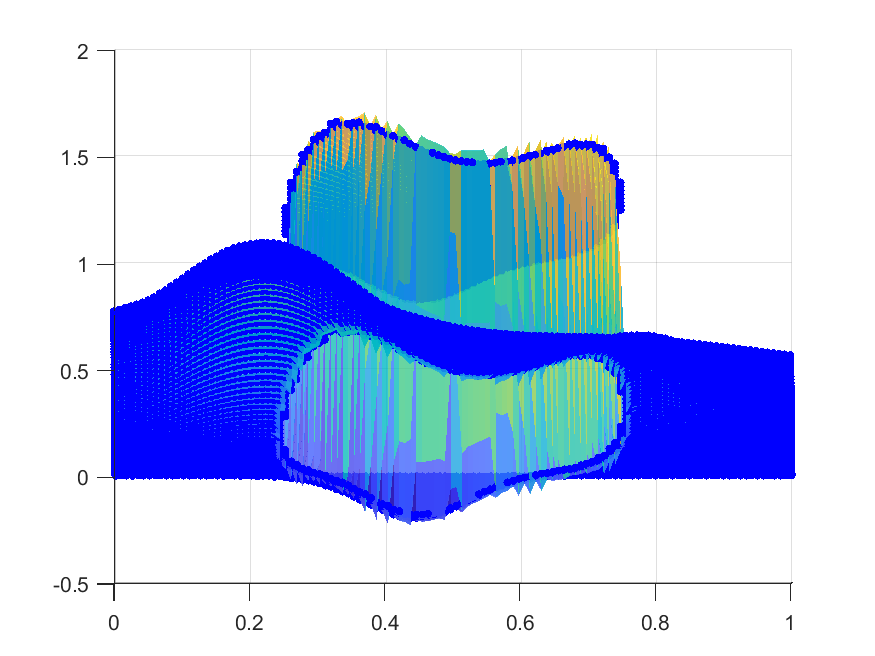} & \includegraphics[width=0.3\hsize]{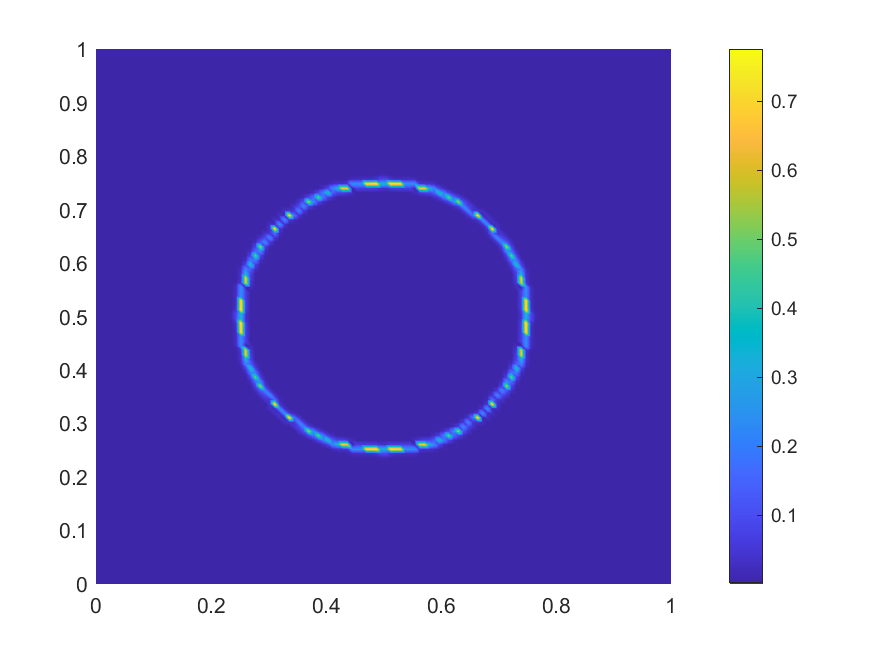}\\
\multicolumn{3}{c}{\scriptsize{{\color{black} DDPU}-MLS}}\\
         \includegraphics[width=0.3\hsize]{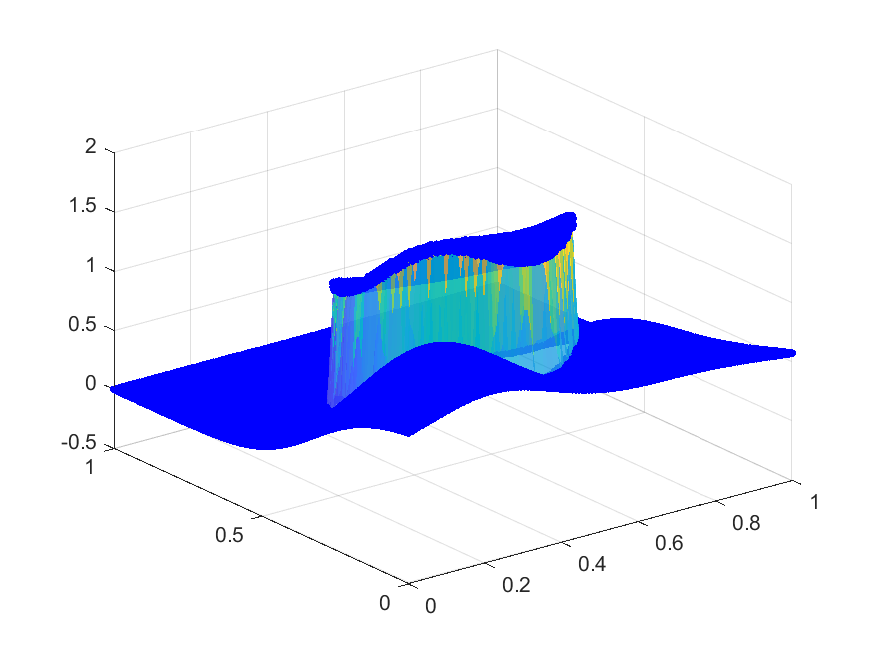} & \includegraphics[width=0.3\hsize]{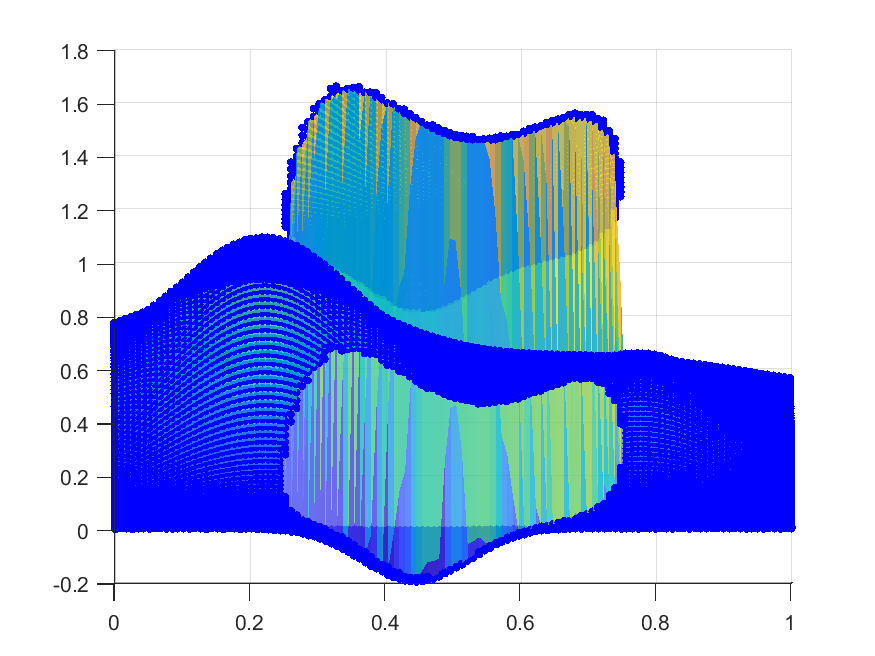} & \includegraphics[width=0.3\hsize]{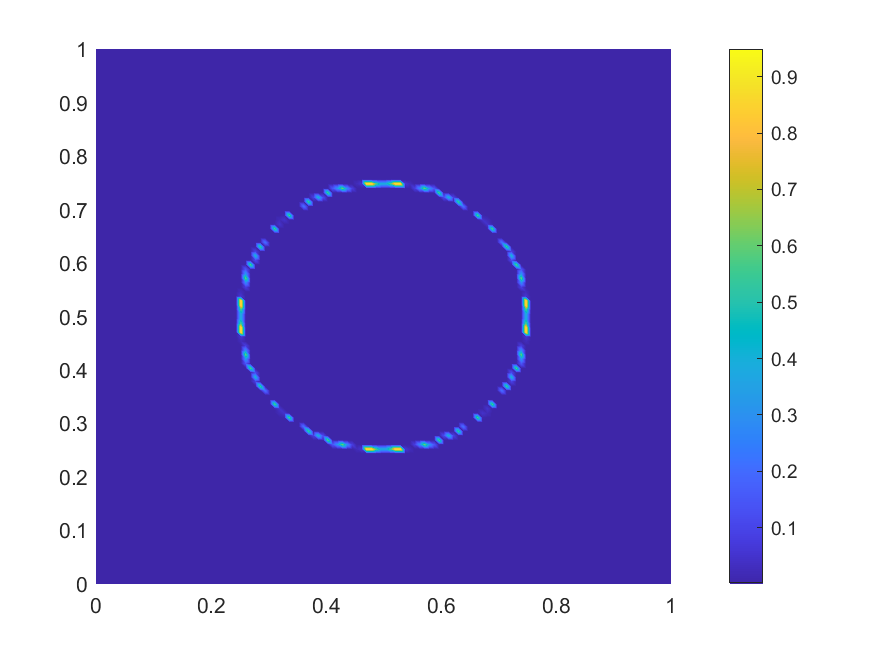}\\
 \end{tabular}
 \caption{\scriptsize{Approximation of the function {\color{black} $f_2$}, Eq. \eqref{frank}, using PU-MLS and {\color{black} DDPU}-MLS with the Wendland $\mathcal{C}^4$ function. The second column shows rotated views of the plots in the first column. The third column shows the errors between the original function and its approximation.}}
 \label{example1W4}
 \end{figure}

 \begin{figure}[H]
\centering
    \begin{tabular}{ccc}
\multicolumn{3}{c}{\scriptsize{PU-MLS}}\\
         \includegraphics[width=0.3\hsize]{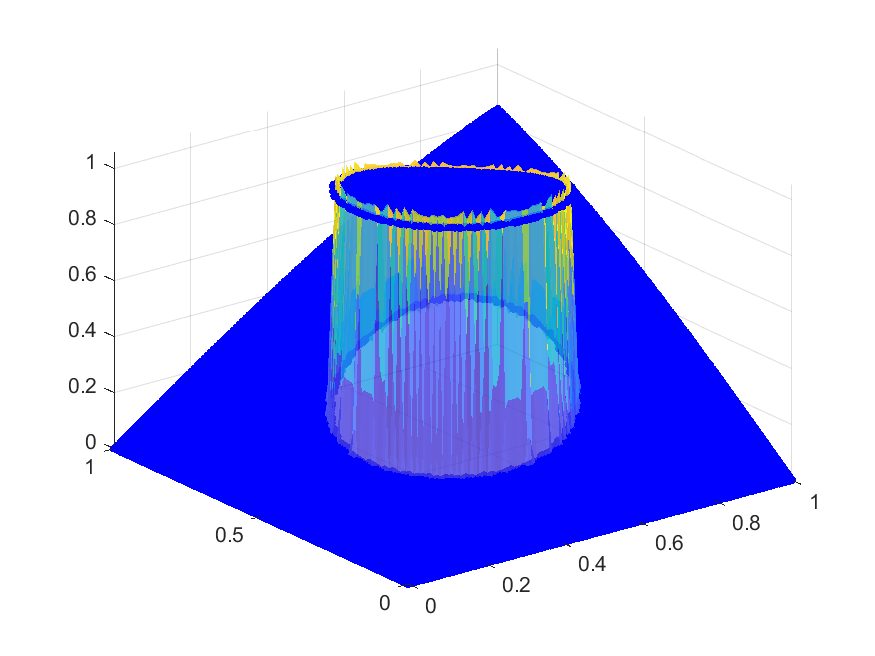} & \includegraphics[width=0.3\hsize]{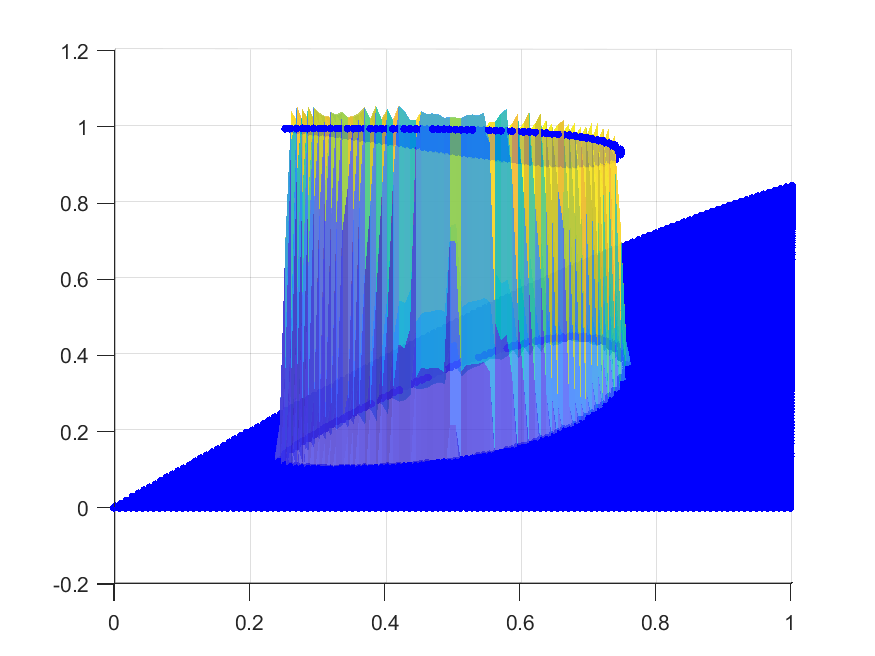} & \includegraphics[width=0.3\hsize]{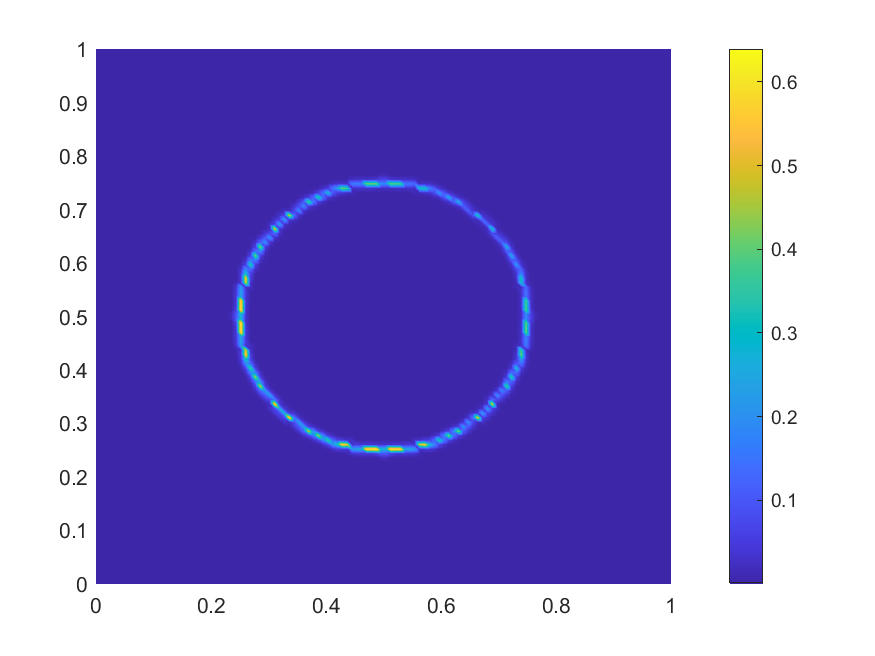}\\
\multicolumn{3}{c}{\scriptsize{{\color{black} DDPU}-MLS}}\\
         \includegraphics[width=0.3\hsize]{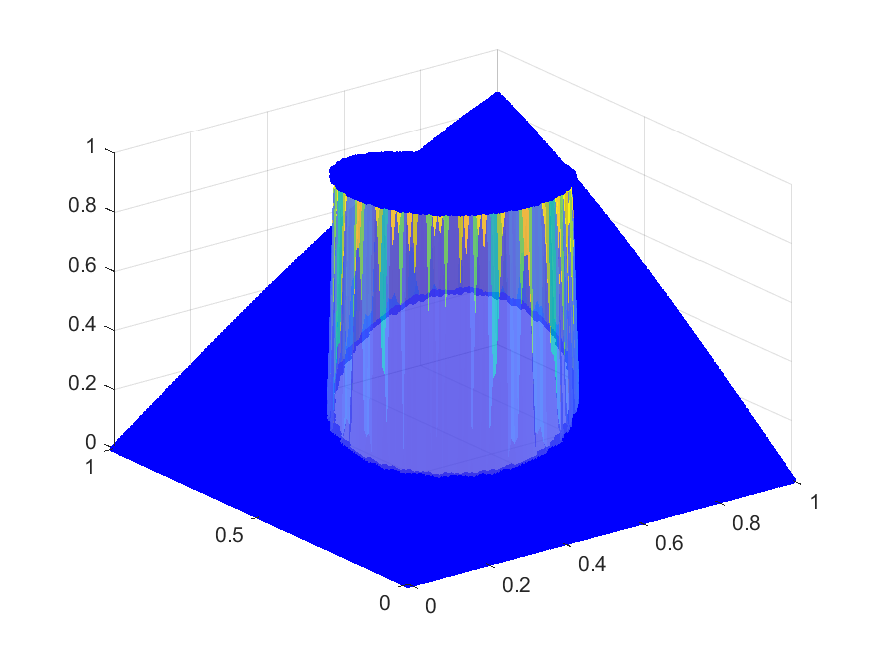} & \includegraphics[width=0.3\hsize]{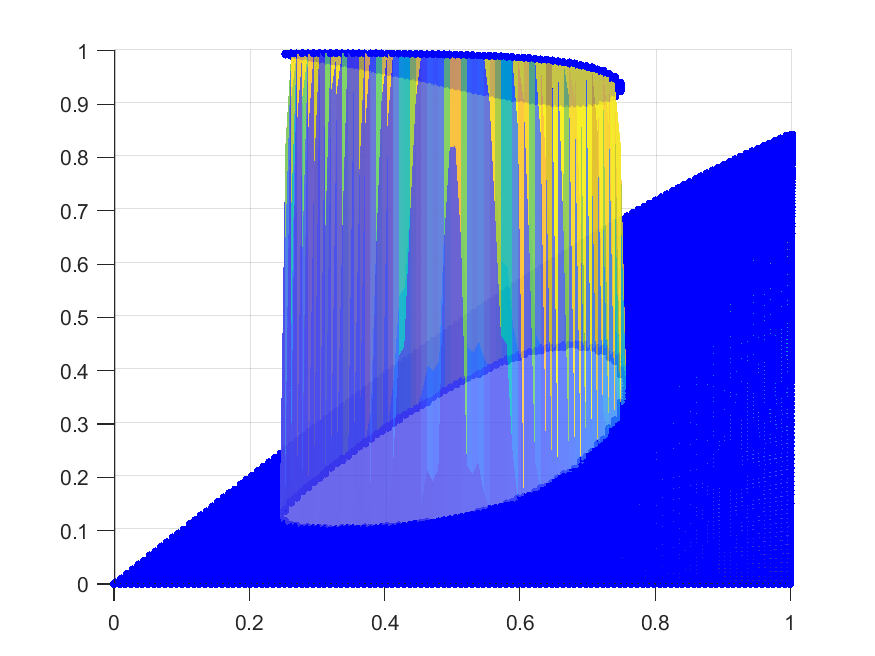}  & \includegraphics[width=0.3\hsize]{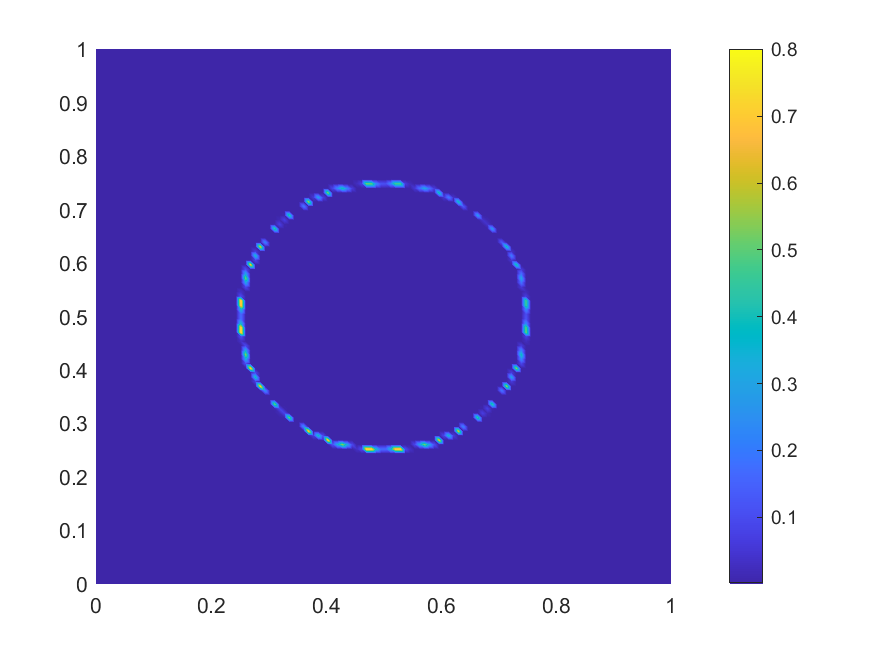}\\
 \end{tabular}
 \caption{\scriptsize{Approximation of the function $g$, Eq. \eqref{Trig}, using PU-MLS and {\color{black} DDPU}-MLS with the Wendland $\mathcal{C}^2$ function. The second column shows rotated views of the plots in the first column. The third column shows the errors between the original function and its approximation.}}
 \label{example2W2}
 \end{figure}

 \begin{figure}[H]
\centering
    \begin{tabular}{ccc}
\multicolumn{3}{c}{\scriptsize{PU-MLS}}\\
         \includegraphics[width=0.3\hsize]{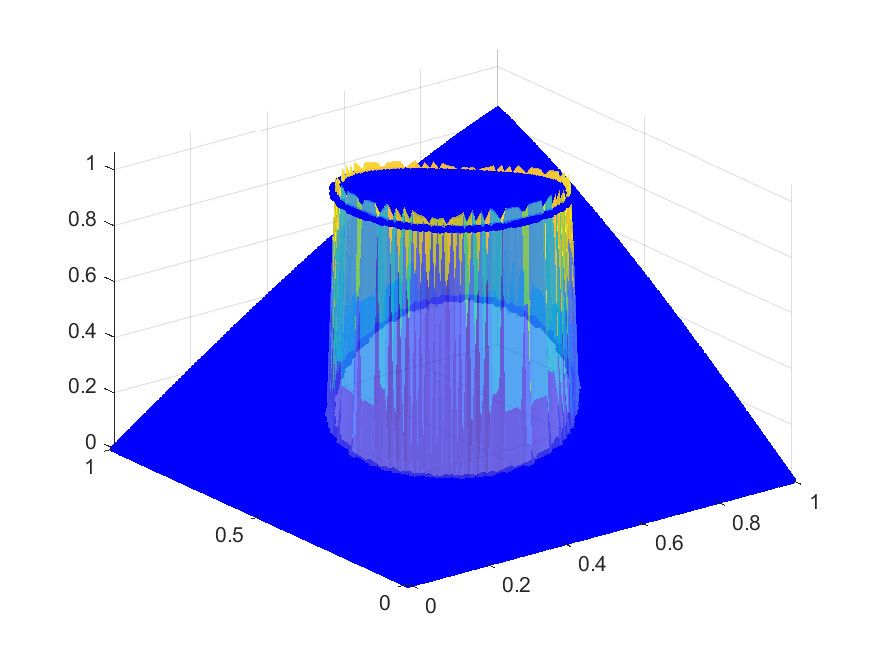} & \includegraphics[width=0.3\hsize]{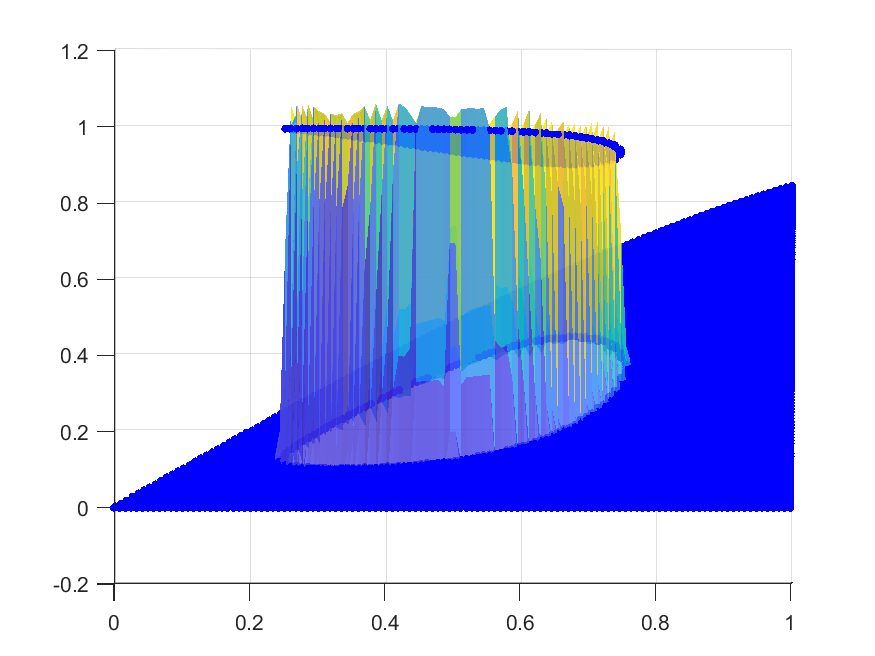} & \includegraphics[width=0.3\hsize]{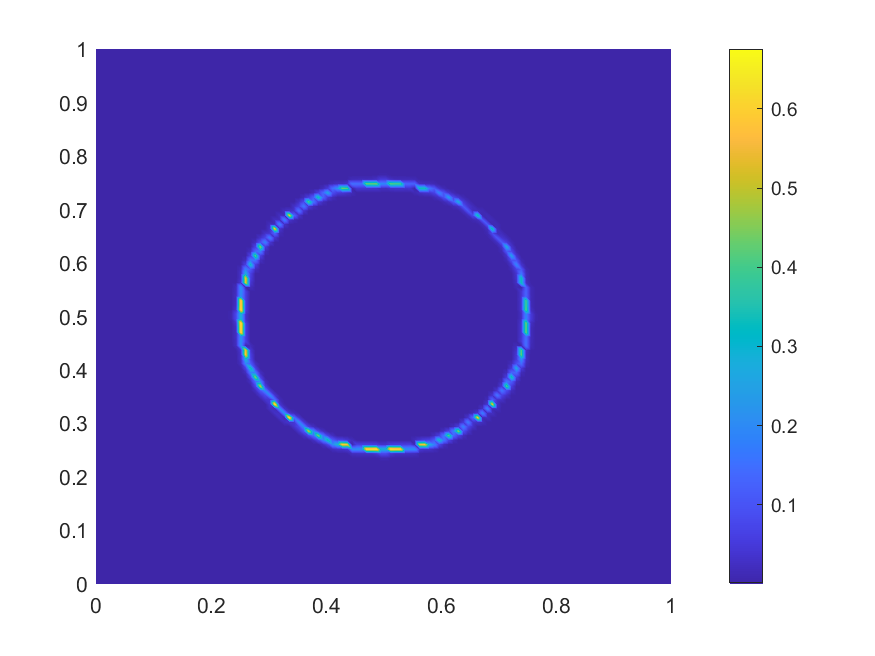}\\
\multicolumn{3}{c}{\scriptsize{{\color{black} DDPU}-MLS}}\\
         \includegraphics[width=0.3\hsize]{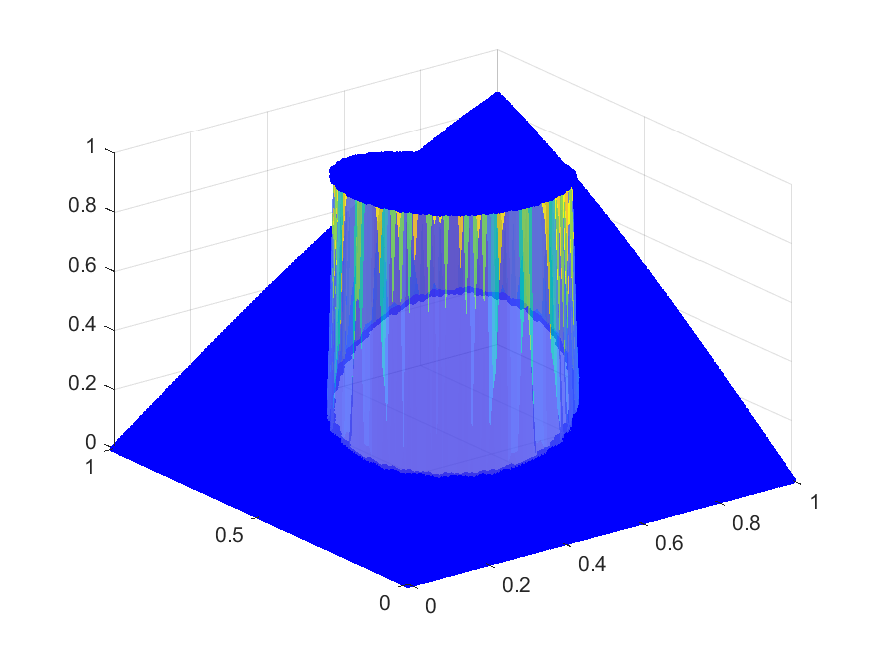} & \includegraphics[width=0.3\hsize]{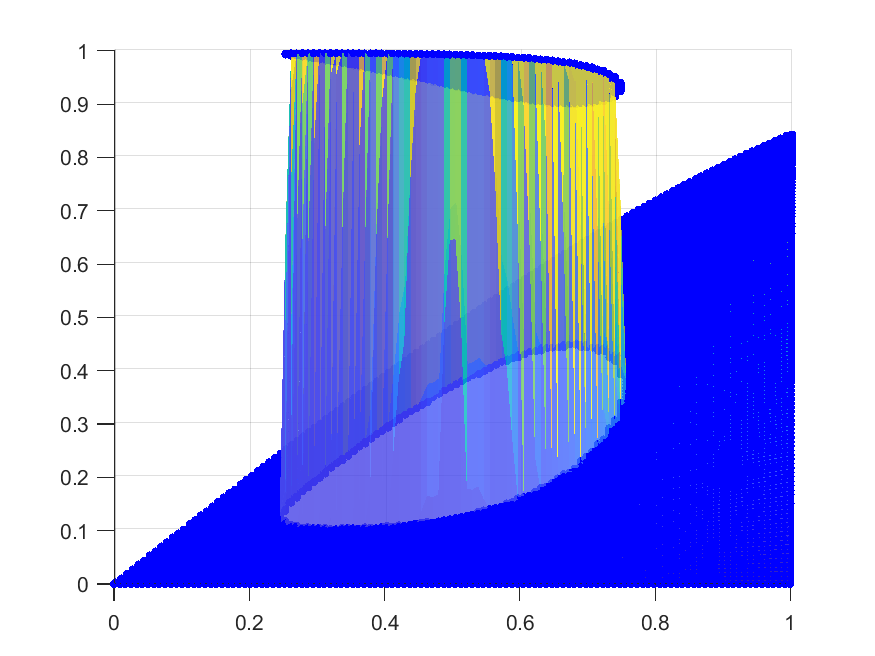} & \includegraphics[width=0.3\hsize]{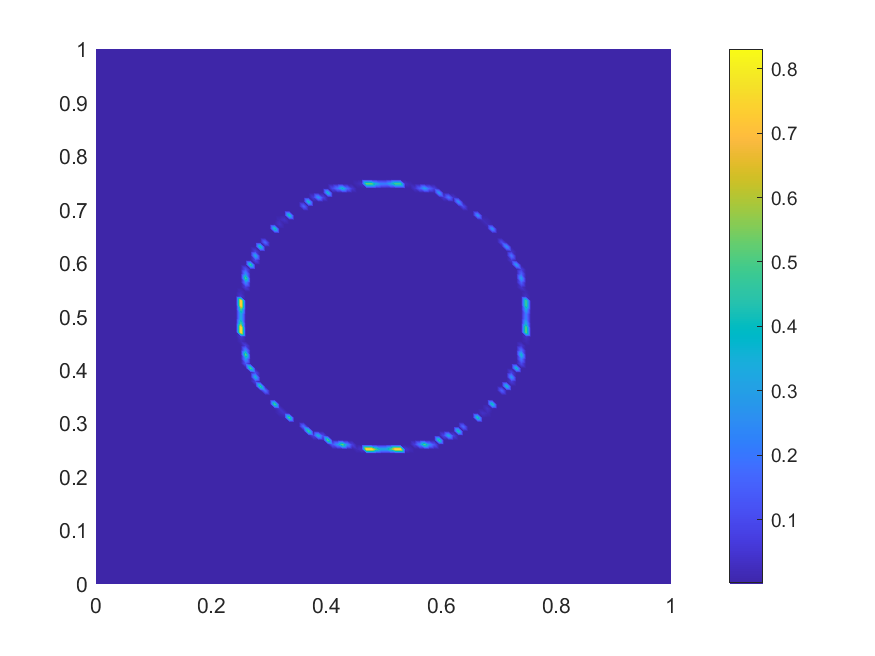}\\
 \end{tabular}
 \caption{\scriptsize{Approximation of the function $g$, Eq. \eqref{Trig}, using PU-MLS and {\color{black} DDPU}-MLS with the Wendland $\mathcal{C}^4$ function. The second column shows rotated views of the plots in the first column. The third column shows the errors between the original function and its approximation.}}
 \label{example2W4}
 \end{figure}

 \begin{figure}[H]
\centering
    \begin{tabular}{ccc}
\multicolumn{3}{c}{\scriptsize{PU-MLS}}\\
         \includegraphics[width=0.3\hsize]{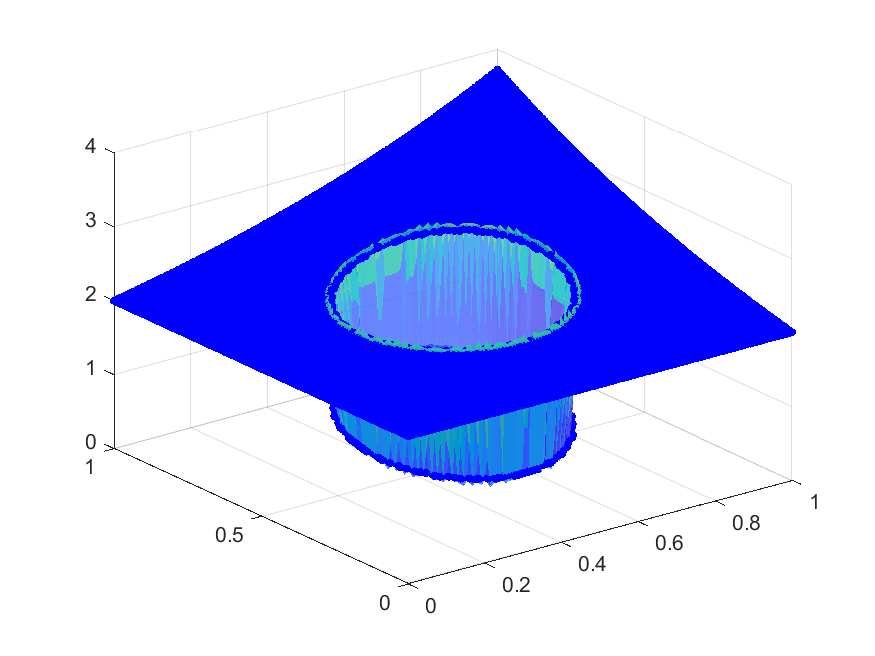} & \includegraphics[width=0.3\hsize]{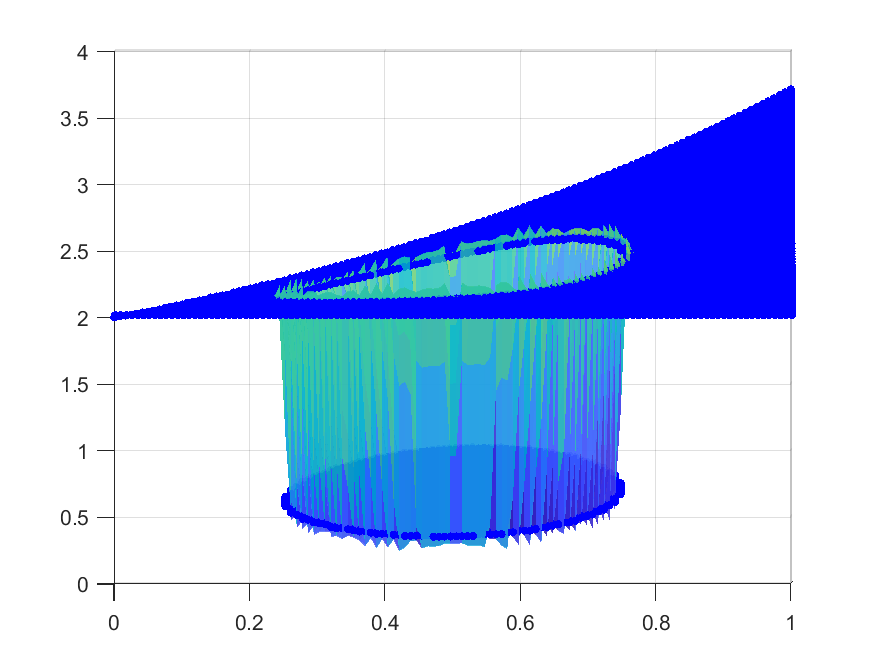} & \includegraphics[width=0.3\hsize]{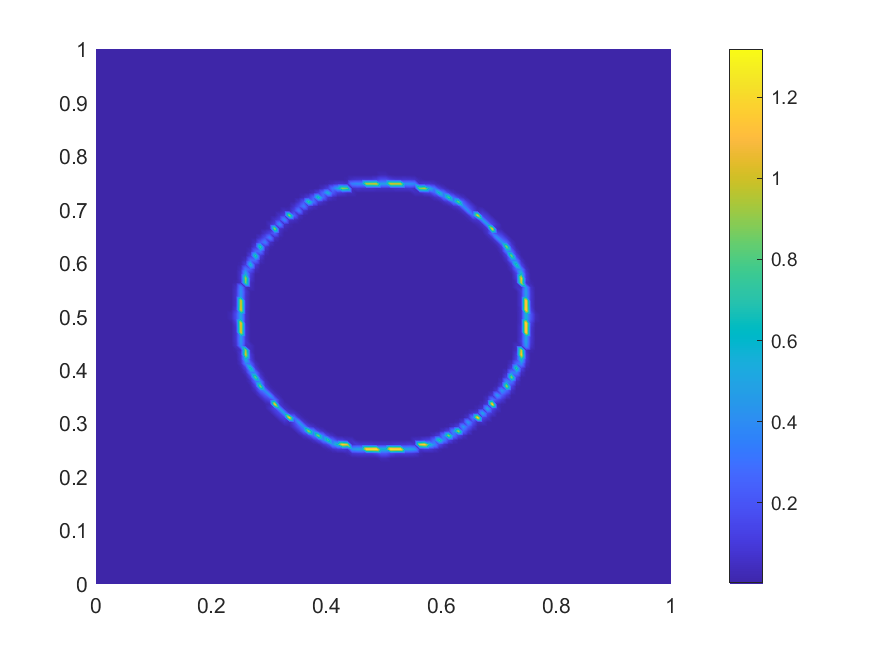}\\
\multicolumn{3}{c}{\scriptsize{{\color{black} DDPU}-MLS}}\\
         \includegraphics[width=0.3\hsize]{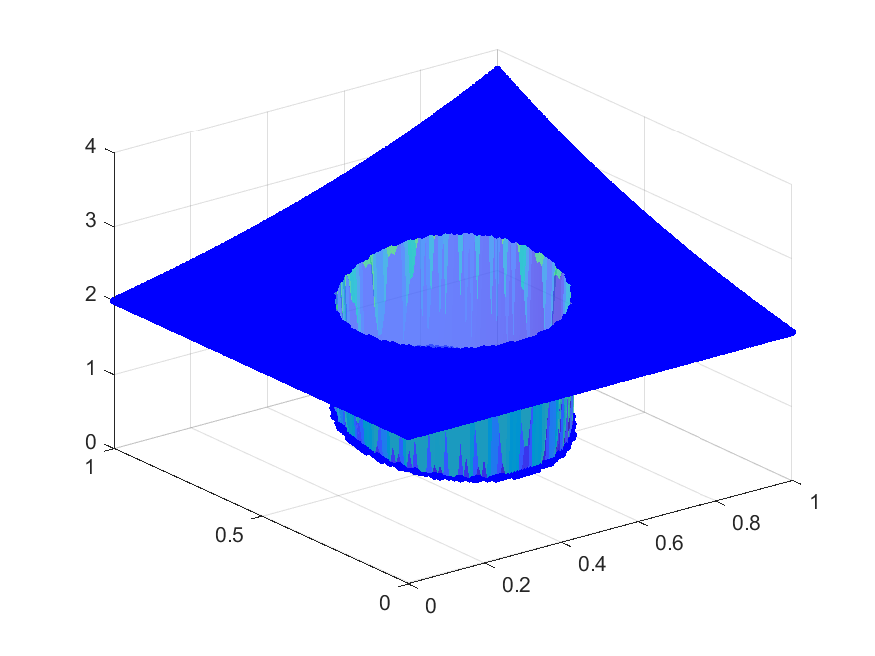} & \includegraphics[width=0.3\hsize]{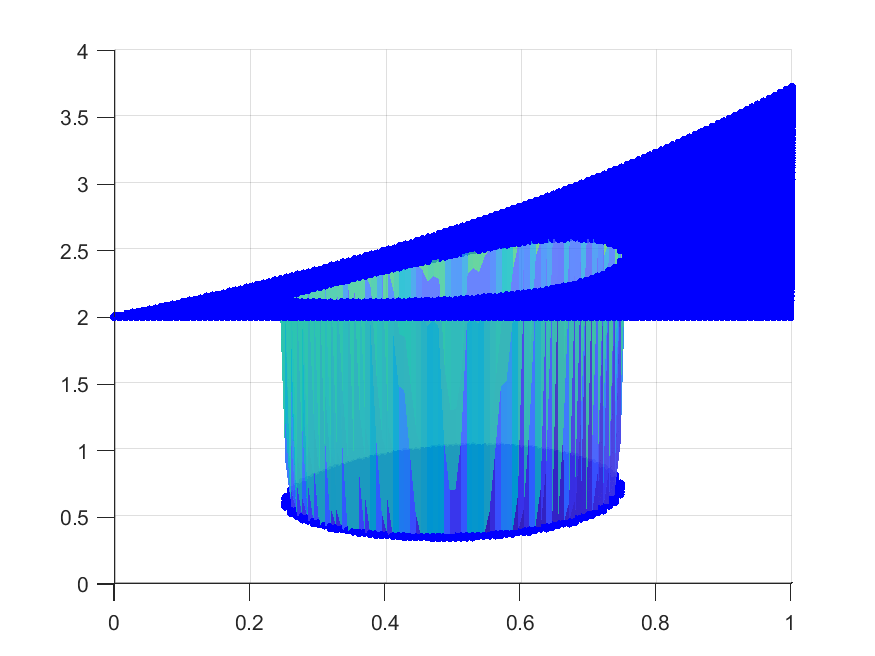} & \includegraphics[width=0.3\hsize]{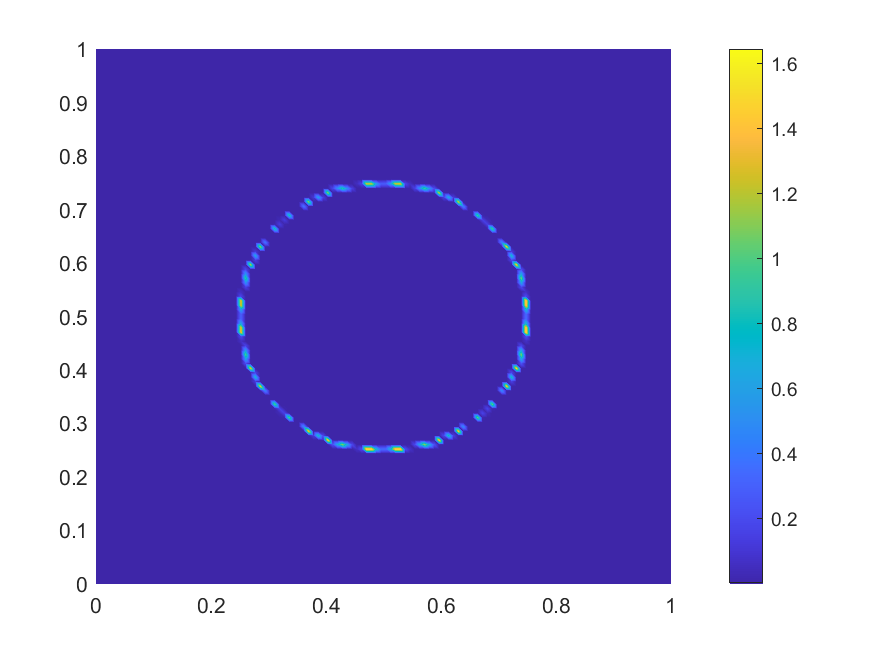}\\
 \end{tabular}
 \caption{\scriptsize{Approximation of the function $h$, Eq. \eqref{Trig2}, using PU-MLS and {\color{black} DDPU}-MLS with the Wendland $\mathcal{C}^2$ function. The second column shows rotated views of the plots in the first column. The third column shows the errors between the original function and its approximation.}}
 \label{example3W2}
 \end{figure}

 \begin{figure}[H]
\centering
    \begin{tabular}{ccc}
\multicolumn{3}{c}{\scriptsize{PU-MLS}}\\
         \includegraphics[width=0.3\hsize]{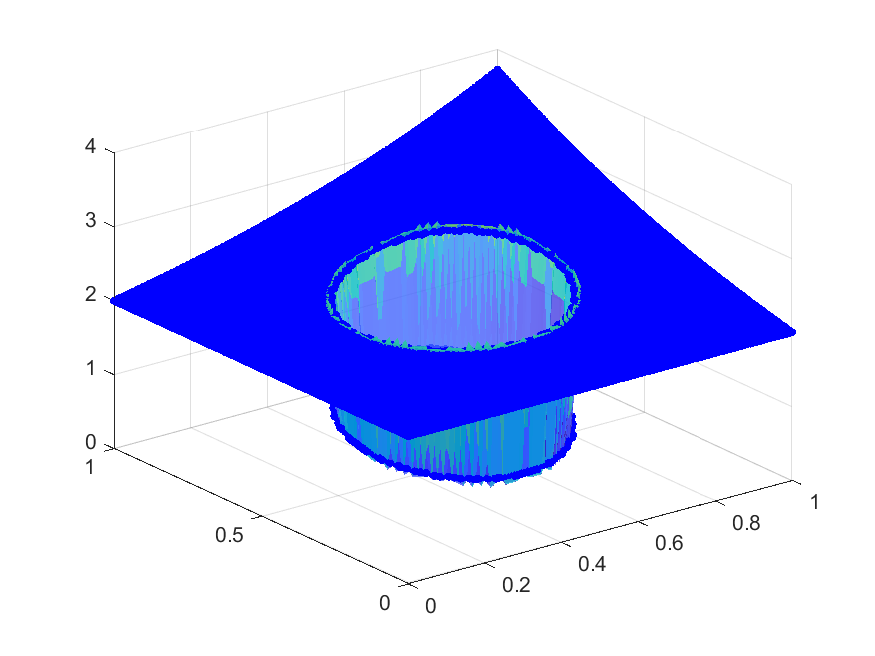} & \includegraphics[width=0.3\hsize]{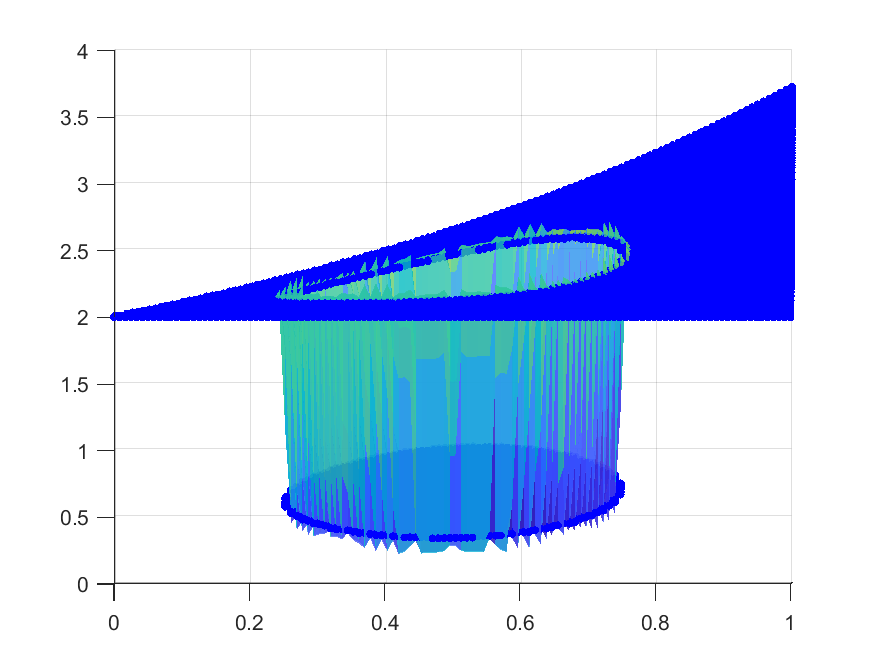}  & \includegraphics[width=0.3\hsize]{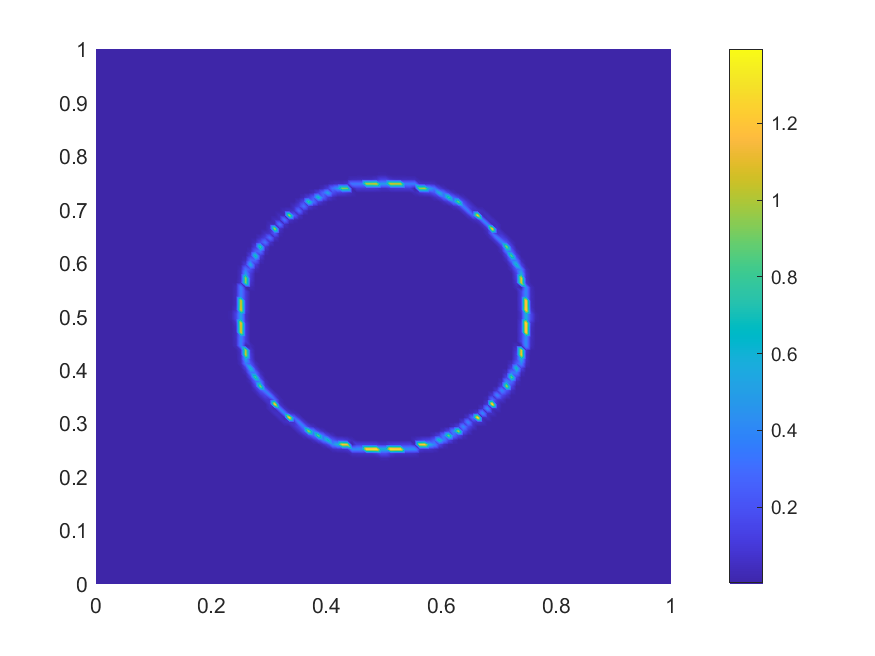}\\
\multicolumn{3}{c}{\scriptsize{{\color{black} DDPU}-MLS}}\\
         \includegraphics[width=0.3\hsize]{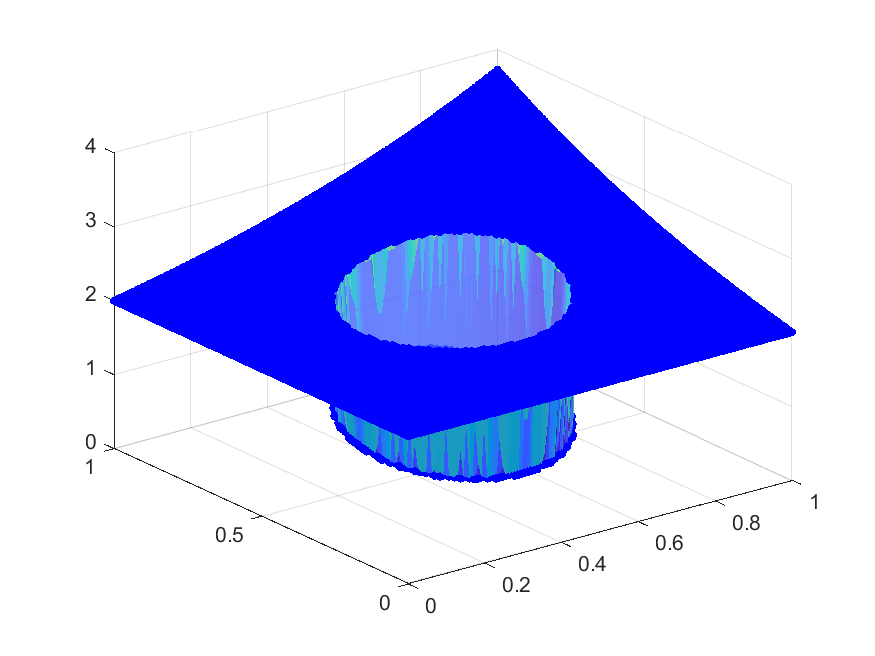} & \includegraphics[width=0.3\hsize]{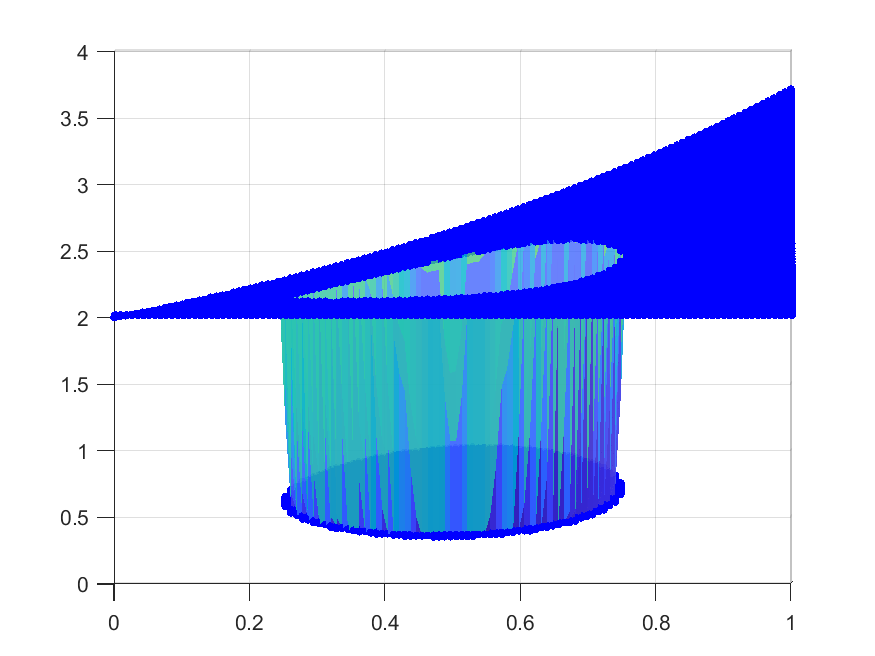} & \includegraphics[width=0.3\hsize]{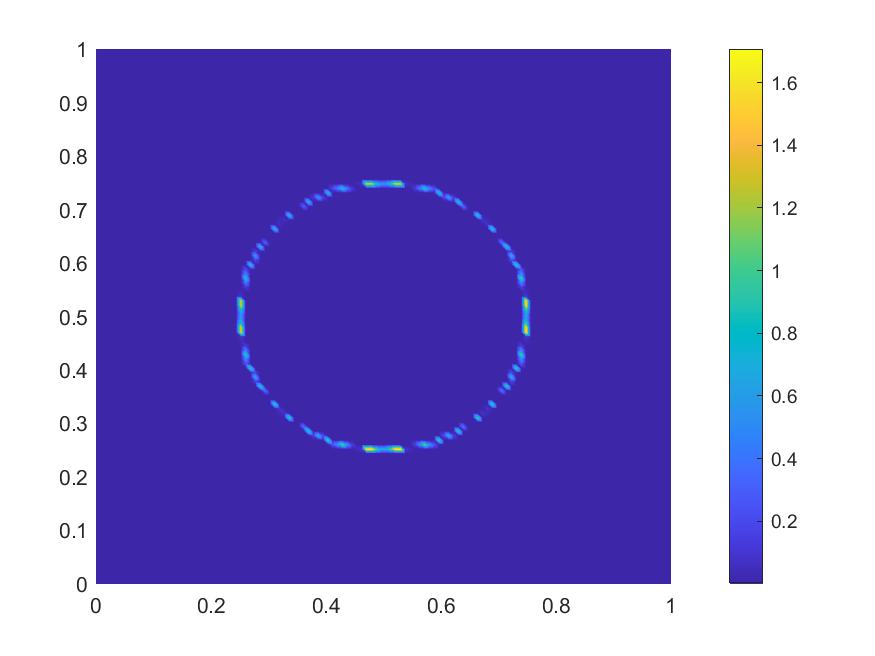}\\
 \end{tabular}
 \caption{\scriptsize{Approximation of the function $h$, Eq. \eqref{Trig2}, using PU-MLS and {\color{black} DDPU}-MLS with the Wendland $\mathcal{C}^4$ function. The second column shows rotated views of the plots in the first column. The third column shows the errors between the original function and its approximation.}}
 \label{example3W4}
 \end{figure}

   \begin{figure}[H]
\centering
    \begin{tabular}{ccc}
\multicolumn{3}{c}{\scriptsize{PU-MLS}}\\
         \includegraphics[width=0.3\hsize]{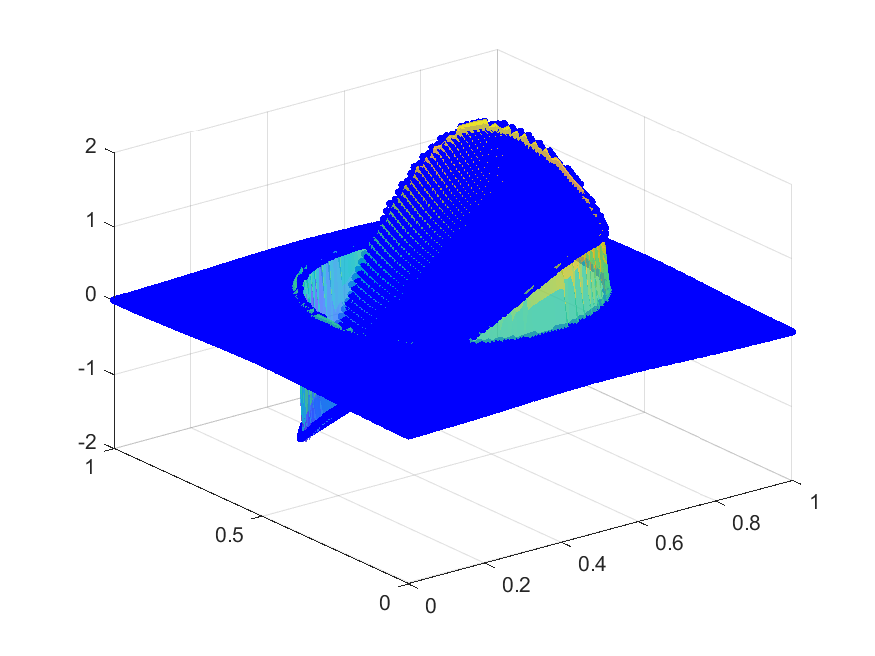} & \includegraphics[width=0.3\hsize]{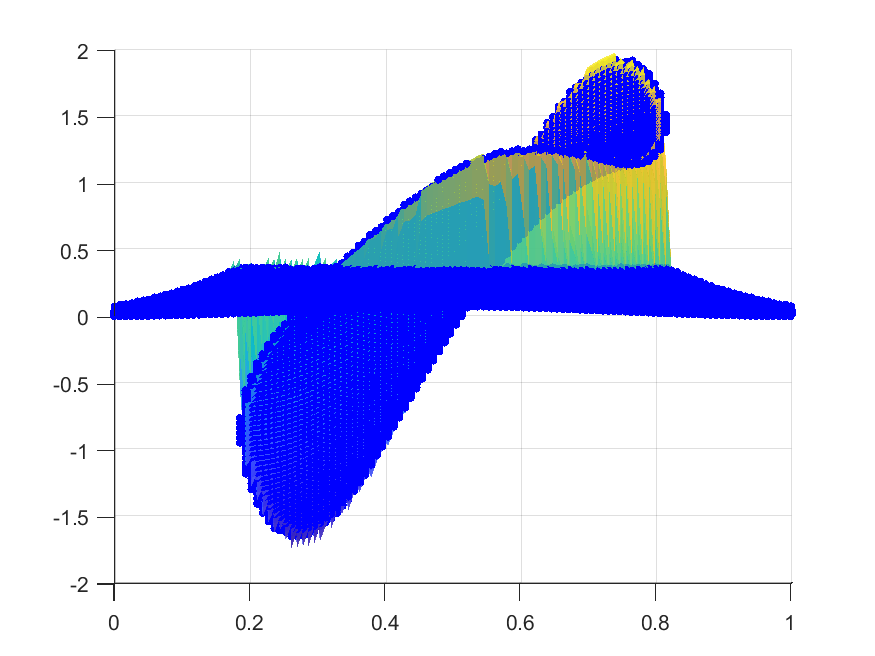} & \includegraphics[width=0.3\hsize]{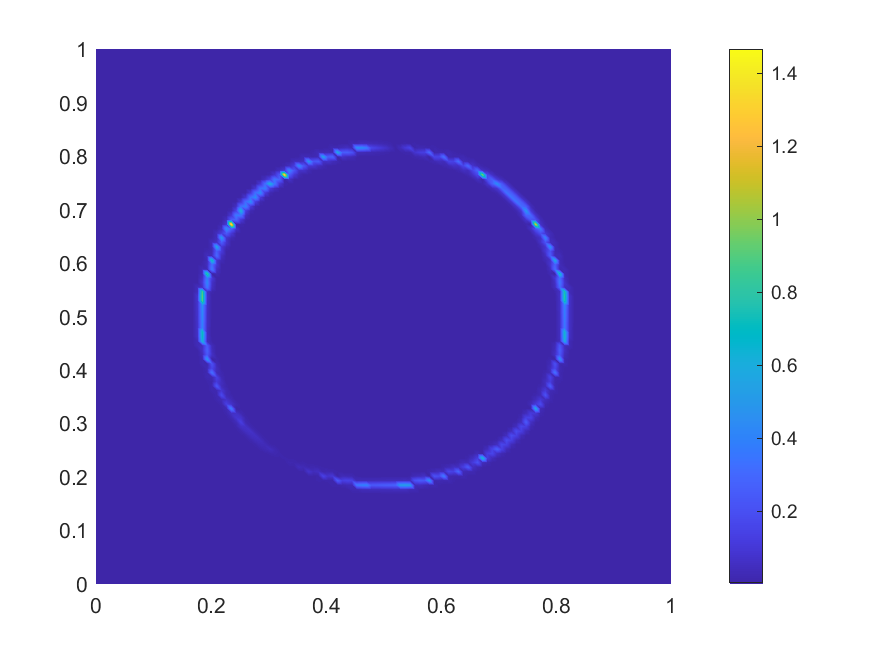}\\
\multicolumn{3}{c}{\scriptsize{{\color{black} DDPU}-MLS}}\\
         \includegraphics[width=0.3\hsize]{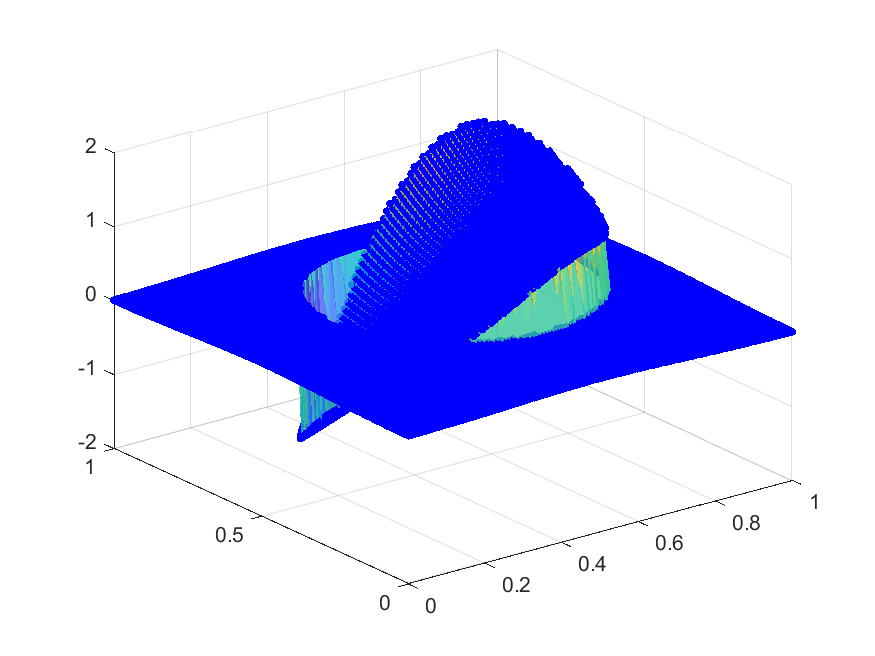} & \includegraphics[width=0.3\hsize]{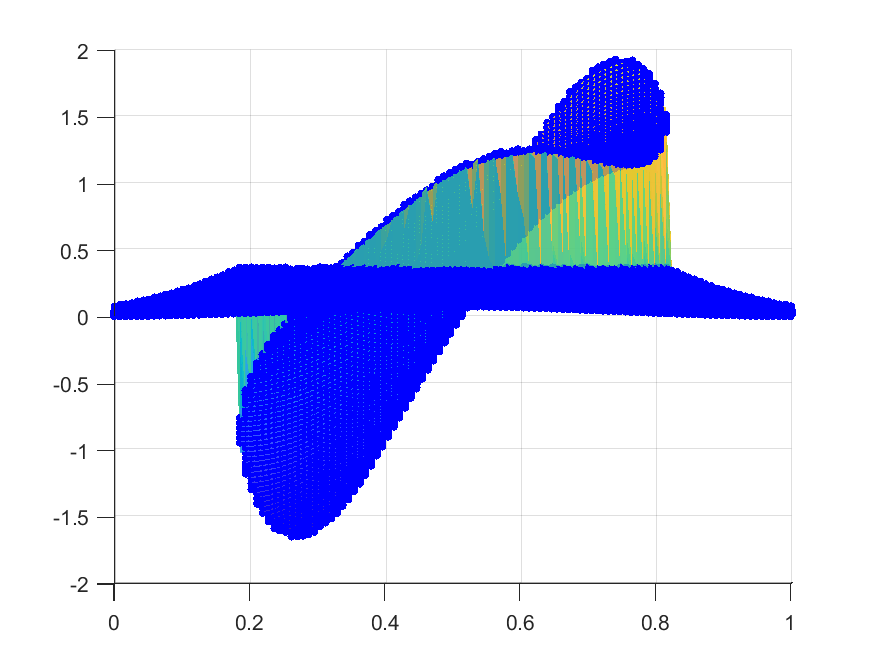} & \includegraphics[width=0.3\hsize]{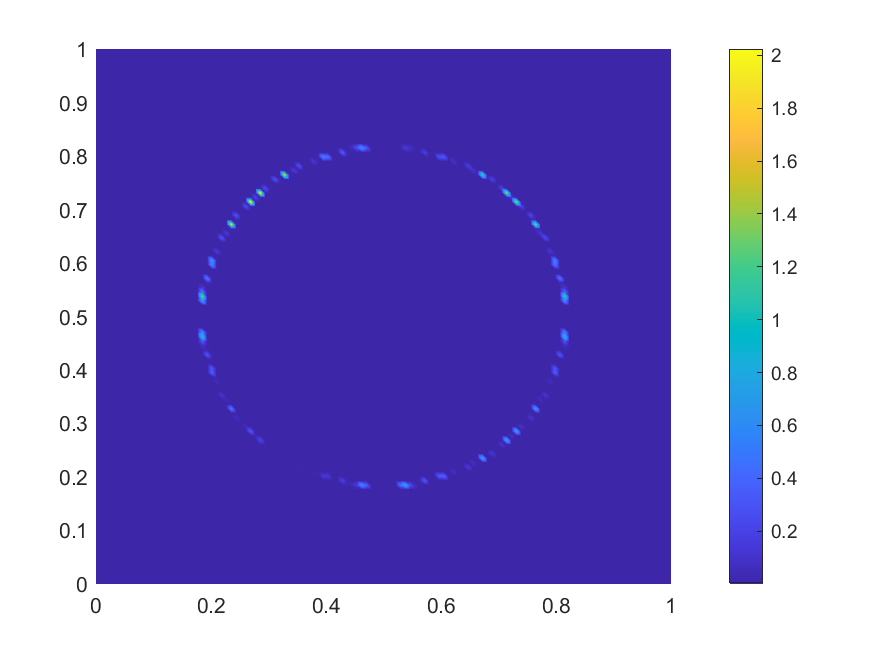}\\
 \end{tabular}
 \caption{\scriptsize{Approximation of the function $j$, Eq. \eqref{Trig3}, using PU-MLS and {\color{black} DDPU}-MLS with the Wendland $\mathcal{C}^2$ function. The second column shows rotated views of the plots in the first column. The third column shows the errors between the original function and its approximation.}}
 \label{example4W2}
 \end{figure}

 \begin{figure}[H]
\centering
    \begin{tabular}{ccc}
\multicolumn{3}{c}{\scriptsize{PU-MLS}}\\
         \includegraphics[width=0.3\hsize]{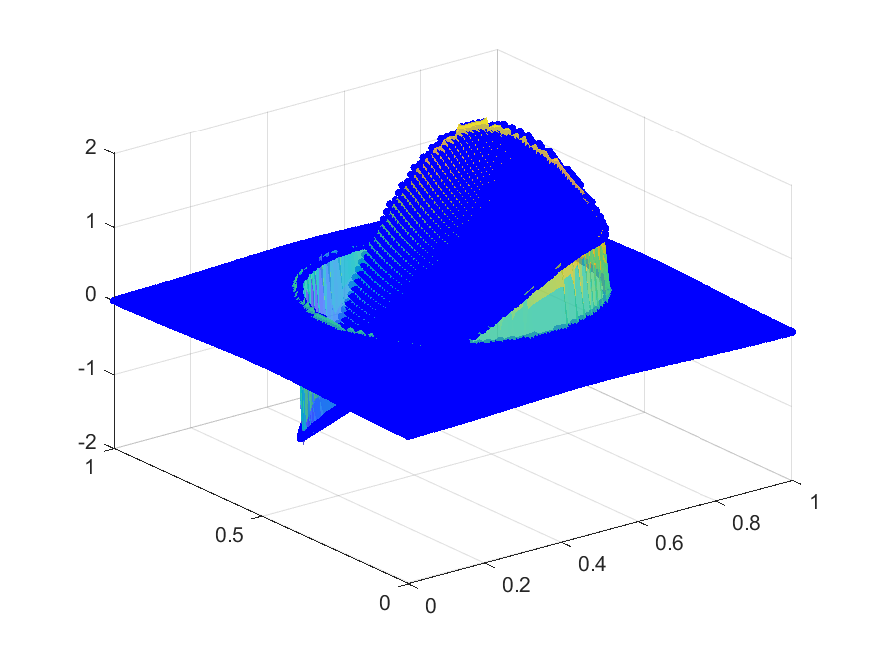} & \includegraphics[width=0.3\hsize]{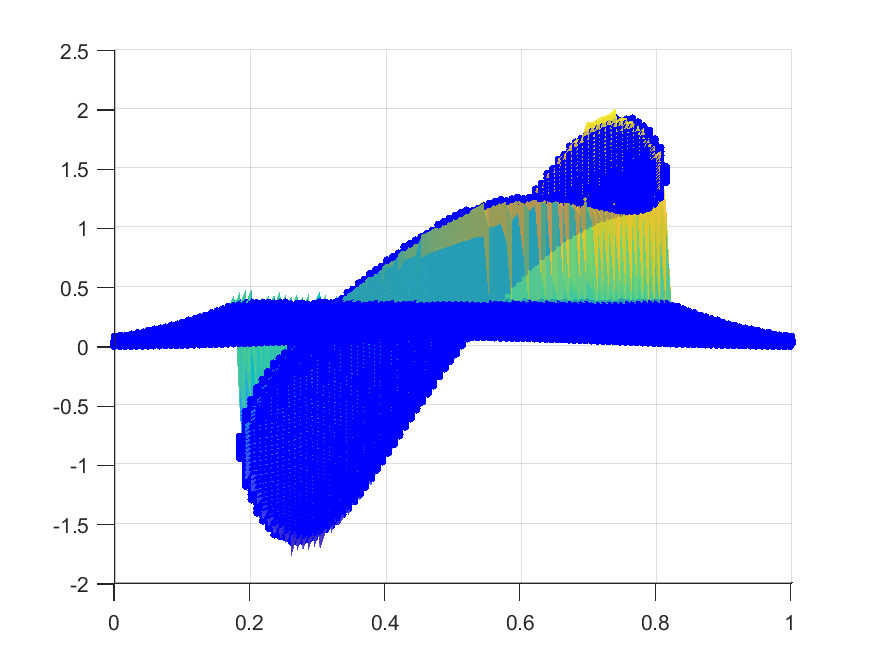}  & \includegraphics[width=0.3\hsize]{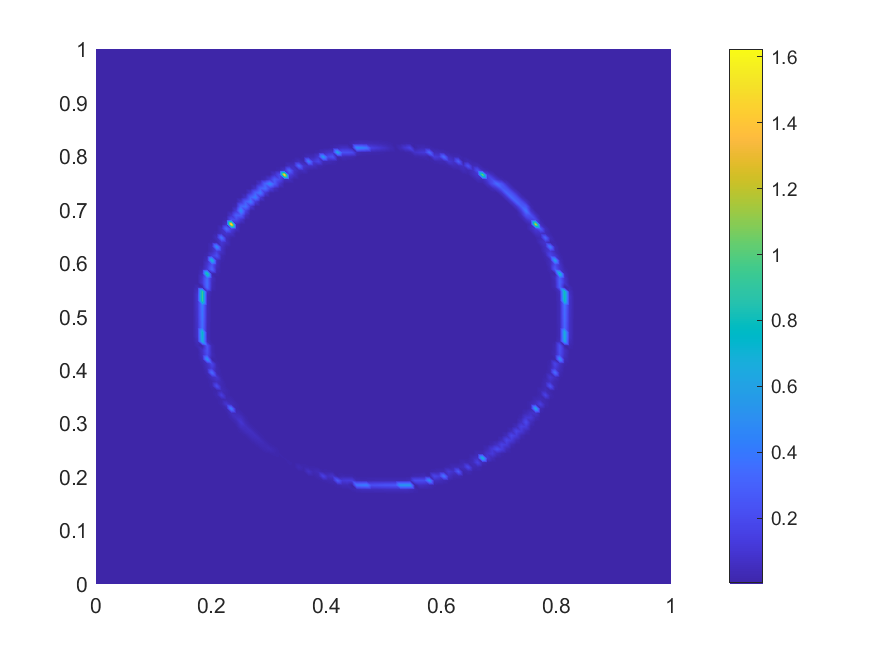}\\
\multicolumn{3}{c}{\scriptsize{{\color{black} DDPU}-MLS}}\\
         \includegraphics[width=0.3\hsize]{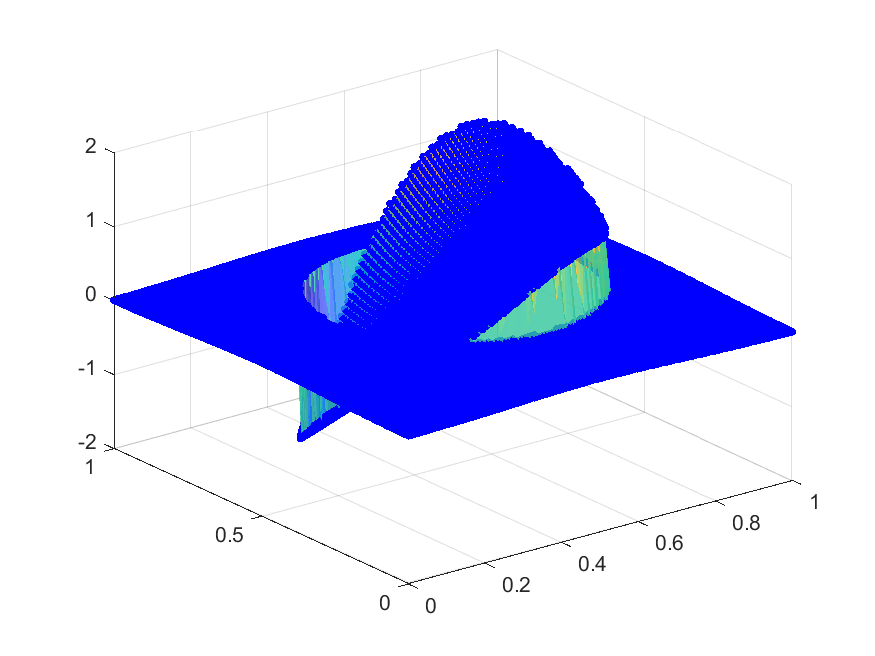} & \includegraphics[width=0.3\hsize]{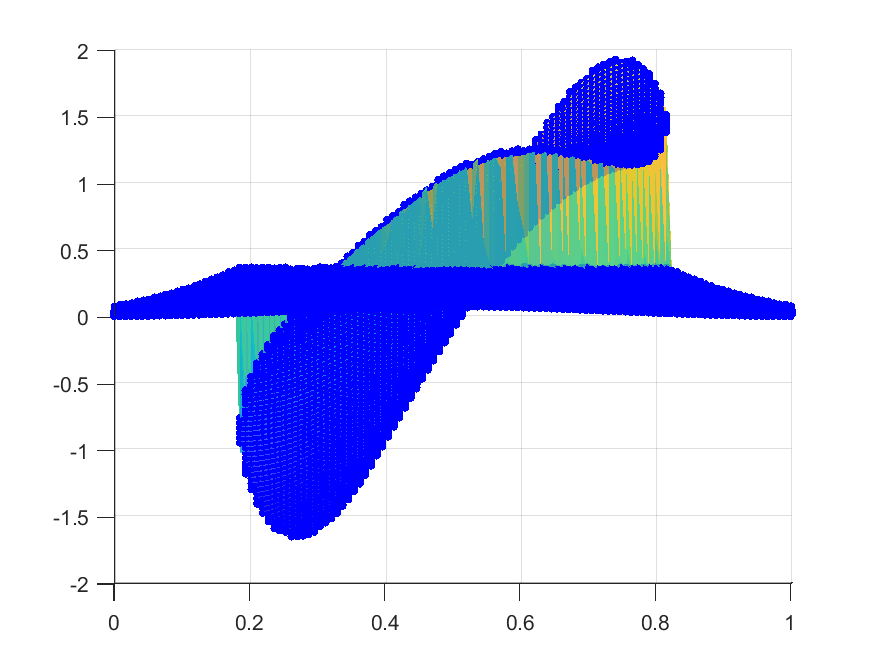} & \includegraphics[width=0.3\hsize]{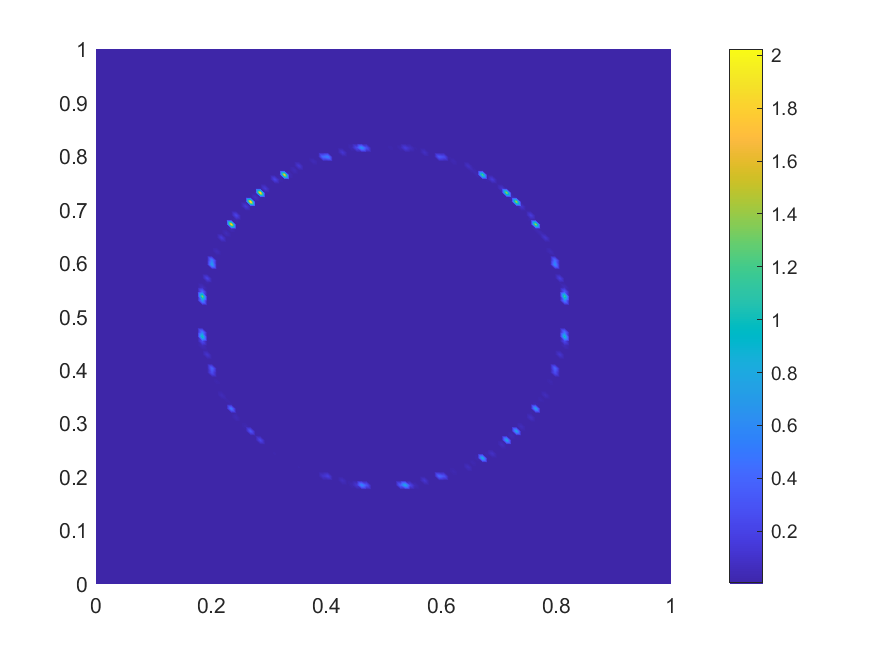}\\
 \end{tabular}
 \caption{\scriptsize{Approximation of the function $j$, Eq. \eqref{Trig3}, using PU-MLS and {\color{black} DDPU}-MLS with the Wendland $\mathcal{C}^4$ function. The second column shows rotated views of the plots in the first column. The third column shows the errors between the original function and its approximation.}}
 \label{example4W4}
 \end{figure}



\end{document}